\documentclass[11pt]{amsart}

\usepackage{graphicx, amsmath, amssymb, relsize, xcolor}
\usepackage{amsmath, amsfonts, amssymb, amsthm, stmaryrd, enumitem, mathrsfs, tcolorbox, float, hyperref}
\usepackage{physics}
\usepackage{dsfont}
\usepackage{amsmath}
\usepackage{tikz}
\usepackage{mathdots}
\usepackage{yhmath}
\usepackage{cancel}
\usepackage{color}
\usepackage{siunitx}
\usepackage{array}
\usepackage{multirow}
\usepackage{amssymb}
\usepackage{gensymb}
\usepackage{tabularx}
\usepackage{extarrows}
\usepackage{booktabs}
\usetikzlibrary{fadings}
\usetikzlibrary{patterns}
\usetikzlibrary{shadows.blur}
\usetikzlibrary{shapes}
\usepackage{parskip}

\date{\today}
\usepackage{geometry}
\usepackage{parskip}
\newtheorem{lemma}{Lemma}[section]
\newtheorem{Lemma}[lemma]{Lemma}
\newtheorem{theorem}[lemma]{Theorem}
\newtheorem{Proposition}[lemma]{Proposition}
\newtheorem{Corollary}[lemma]{Corollary}
\newtheorem{Remark}[lemma]{Remark}

\newtheorem{Hypothesis}[lemma]{Assumption}

\numberwithin{equation}{section}

\newcommand{\bE}{\boldsymbol{E}}
\newcommand{\bK}{\boldsymbol{K}}

\newcommand{\eps}{\varepsilon}

\newcommand{\N}{\mathbb{N}}

\newcommand{\R}{\mathbb{R}}

\newcommand{\Cc}{\mathcal{C}}

\newcommand{\T}{\mathbb{T}}

\newcommand{\Z}{\mathbb{Z}}

\newcommand{\e}{\delta}

\newcommand{\tildeg}{g}

\newcommand{\Norm}[2]{\|#1\|\left.\vphantom{T_{j_0}^0}\!\!\right._{#2}}

\title[Dynamics close to inhomogeneous stationary states of the  Vlasov-HMF model]{Large time nonlinear  dynamics close to inhomogeneous stationary states of the  Vlasov-HMF model}

\author[E. Faou, F. Rousset and T. Seetohul]{Erwan Faou$^1$, Fr\'ed\'eric Rousset$^2$ and  Tooryanand Seetohul$^3$}

\begin{document}

\maketitle

\footnotetext[1]{Univ Rennes, INRIA centre de l'Universit\'e de Rennes (Mingus team) and IRMAR UMR CNRS 6625, F-35042 Rennes, France \& ENS Rennes, France (\texttt{Erwan.Faou@inria.fr)}}

\footnotetext[2]{Laboratoire de Math\'ematiques d'Orsay (UMR 8628), Universit\'e Paris-Saclay, B\^atiment 307, Rue Michel Magat, F-91405 Orsay, France (\texttt{frederic.rousset@universite-paris-saclay.fr)}}

\footnotetext[3]{Univ Rennes, INRIA centre de l'Universit\'e de Rennes (Mingus team) and IRMAR UMR CNRS 6625, F-35042 Rennes, France \& ENS Rennes, France (\texttt{tooryanand.seetohul@univ-rennes.fr)}

{\bf Acknowledgement}: This work was partly supported by the Simons collaboration on wave turbulence, the ANR project KEN, ANR-22-CE40-0016, the Centre de Marth\'ematiques Henri Lebesgue ANR-11-LABX-0020, and the ANR projet BOURGEONS, ANR-23-CE40-0014.}

\begin{abstract}
	We consider the Vlasov-HMF (Hamiltonian Mean-Field) model and symmetric perturbations of symmetric inhomogeneous stationary states that are naturally associated with a pendulum Hamiltonian flow. We first show the existence of families of such stationary states having compact support inside the separatrix region, constant value near the origin, and satisfying a linear stability (Penrose) criterion. 
We then describe the nonlinear  dynamics of small compactly supported perturbations with finite regularity of this class of  equilibria.
More precisely, we prove that for large  (but finite) times,  in a modulated action-angle coordinate system the solution remains close to  a modulated  stationary state  and that the magnetization of the perturbation satisfes  decay estimates driven by  the linear dynamics. 
 \end{abstract}

\section{Introduction}

For $(t, x, v) \in \R^{+} \times \T \times \R$, with $\T = \R \slash 2\pi \Z$, we consider  the Vlasov-HMF (Hamiltonian Mean-Field) model given by
\begin{equation}
\label{eq:VlasovHMFmodel}
\left|
\begin{array}{l}
 \partial_{t} f(t, x, v) + \bigl\{ f, H[f] \bigr\}_{x, v}(t, x, v) = 0,  \\
 H[f](t, x, v) = \frac{v^{2}}{2} - \phi[f](t, x), \\
 \phi[f](t, x) = \int_{\mathbb{T} \times \mathbb{R}} \cos(x-y) f(t, y, w) dy dw, \\
 f(0,x,v) = f^0(x,v)
\end{array}
\right.
\end{equation}
where $ \bigl\{ f, g \bigr\}_{x, v} = \partial_{x} f \partial_{v} g - \partial_{v} g \partial_{x} f$ denotes the Poisson bracket and $\int_{\T} f(x) dx = \frac{1}{2\pi} \int_{0}^{2\pi} f(x) dx $ for a periodic function, and $f^0(x,v)$ denotes the initial datum. 
This model is a toy model for plasma particles and is a simplification of the more central Vlasov-Poisson system where we replace the Poisson potential on the torus   by  a cosine interaction. It remains long-range and preserves many underlying features of  more complex kinetic equations (see \cite{Yamaguchi_2004}, \cite{Chavanis_2005}, \cite{Chavanis_2006}, \cite{PhysRevE.52.2361} and \cite{tsallis2002}). 

 The potential $\phi[f](t, x)$ in \eqref{eq:VlasovHMFmodel} can be decomposed into
\begin{equation}
\phi[f](t, x) =\mathcal{C}[f](t)  \cos(x)  +  \mathcal{S}[f](t) \sin(x), 
\mbox{ where }
\label{zrevsdv}
\left|
\begin{array}{l}
\mathcal{C}[f](t) = \int_{\mathbb{T} \times \mathbb{R}} \cos(y) f(t, y, w) \, dy dw, \\
\mathcal{S}[f](t) = \int_{\mathbb{T} \times \mathbb{R}} \sin(y) f(t, y, w) \, dy dw. 
\end{array}
\right.
\end{equation}
 We  shall consider here  symmetric distribution functions $f(t, x, v) = f(t, -x, -v)$ so  that $\mathcal{S}[f](t) \equiv 0$ for all time, simplifying the potential to 
 \begin{equation}
 \label{potsimple}
 \phi[f](t, x) = \mathcal{C}[f](t)  \cos(x).
 \end{equation} 
 The scalar quantity $\mathcal{C}[f](t)$ is often called the magnetization in the HMF framework.
 
 In this paper, we shall be interested in the asymptotic  nonlinear stability of space inhomogeneous stationary states, a property often called nonlinear Landau damping.  
 
 Nonlinear Landau damping around homogeneous stationary states on the torus is now well understood for  general Vlasov equations 
 including Vlasov Poisson since the breakthrough work of C. Mouhot and C. Villani \cite{Mouhot_2011} and the subsequent simplifications and improvements
 \cite{bedrossian2013landaudampingparaproductsgevrey},  \cite{grenier2022landaudampinganalyticgevrey}, \cite{Ionescu-Pausader-Gevrey}. 
 We also refer to \cite{Caglioti}, \cite{Iacobelli} for the related problem of the contruction of solutions with prescribed asymptotic behavior.
  The recent  work   \cite{Ionescu-Pausader-Gevrey} in particular handles the critical Gevrey regularity for Vlasov-Poisson. In general, Gevrey regularity seems necessary for nonlinear Landau damping  on the torus, \cite{Bedrossian-instability},  due to the presence of echoes (the case of negatively curved compact manifold studied in \cite{Han-Kwan-Riviere} is different). 
 Nevertheless, in the case of the HMF model on the torus, the nonlinearity can be seen as completely non-resonant, we have obtained  in \cite{Faou_2015} nonlinear Landau damping in finite Sobolev regularity.
 
  Note that the case of the whole space $x \in \mathbb{R}^d$ is now also well-understood for screened interactions \cite{Bedrossian-Mouhot-Masmoudi}, 
  \cite{Han-Kwan-Nguyen-Rousset-screened}, \cite{Vietnamiens}, \cite{Wei} but still under investigation for the unscreened Vlasov-Poisson equation in dimension $3$ where only linear results  \cite{Bedrossian-Mouhot-linear}, \cite{Han-Kwan-Nguyen-Rousset-linear} \cite{Toanlin} or nonlinear results for specific equilibria with slow  decay  \cite{Ionescu-Pausader-Poisson} are available. 
  
In the  case of inhomogeneous states on the torus for various classes of  Vlasov equations, significant progress have been made in the study of linear stability,  we can mention  \cite{Faou2021-qm} for the HMF model,   \cite{Despres},  \cite{Weder}, \cite{Grenier-Pausader}   for the repulsive  Vlasov Poisson  and  in the gravitational case \cite{Hadzic1}, \cite{Hadzic2}.
The nonlinear Lyapounov stability of equilibria has been investigated for some classes of equilibria in \cite{fontaine2017stablegroundstateshmf} for the HMF model by using techniques developped for the gravitational Vlasov-Poisson system  \cite{Lemou-Raphael}.

In the nonlinear case, the related problem of the asymptotic stability of the zero solution in an external potential is studied in  \cite{ChaLuk24}  for trapped particles, leading to difficulties closely related to those encountered in the present paper, and in  \cite{Pausader-point}  for  untrapped particles.
Our aim here is to initiate the study of the asymptotic nonlinear stability of inhomogeneous states with trapped particles, focusing on the simplest case, namely the HMF model which is the most favorable one  for  homogeneous states stability  in view of the results in [15] mentioned above.

For the HMF model,  inhomogeneous and symmetric stationary states $\eta(x, v) = \eta(-x, -v)$  solve  $\bigl\{ \eta, H[\eta] \bigr\}_{x, v} = 0$. 
To find a stationary state one can thus prescribe that
\begin{equation}
\label{eqh0} 
  \eta(x,v) = G(h_{M}(x,v)), \quad G: \R \to \R,  \quad \mbox{ with }\quad h_{M}(x, v) := \frac{v^{2}}{2}  - M \cos(x)
\end{equation}
for some smooth function $G$ for some time independent  $M$ which is the magnetization of the stationary state. This yields a nontrivial  non-homogeneous stationary state if one can then  find
$M>0$ such that
\begin{equation}
H[\eta](x,v) = h_{M}(x, v) \quad \Longleftrightarrow \quad \label{lequationsurM}
M = \Cc[\eta] = \int_{\T \times \R} \cos(x) G \Big(\frac{v^2}{2} - M \cos(x)\Big)   \, dx \, dv. 
\end{equation}

Some examples of stationary states are given in \cite{Faou2021-qm}  where  we have studied the linearized dynamics  which is given (again by restricting to even solutions) by 
\begin{equation}
\label{linearizedintro}
\partial_{t}g + \left\{ g, h_{M}\right\}  - \mathcal{C}[g](t)\left\{\eta_{M} , \cos x\right\}=0, \quad g_{/t=0}=g_{0}.
\end{equation}
The main results of \cite{Faou2021-qm} are on the one hand the study of the mixing properties of the linear transport given by
\begin{equation}
\label{freependule}
\partial_{t}f + \left\{ f, h_{M}\right\}= 0, \quad f_{/t=0}=f^{0} \quad \Longleftrightarrow \quad 
f(t,x,v) = f^0(\varphi_M^{-t}(x,v))
\end{equation}
where $\varphi^t_M(x,v)$ denotes the flow of the Hamiltonian $h_M$.
$$
(x(t),v(t)) = \varphi^t_M(x,v) 
\quad \Longleftrightarrow \quad \left|
\begin{array}{l}
\dot x(t) = v(t)  \\
\dot v(t) = - M \sin x(t) 
\end{array}
\right.
\quad \mbox{with} \quad
\left|
\begin{array}{l}
 x(0) = x,  \\
v(0) = v. 
\end{array}
\right.
$$
This latter Hamiltonian is known to model the oscillations of a pendulum in classical mechanics. It is completely integrable and allows for action-angle coordinates $(\theta, a)$.  
\begin{figure}
\label{fig1}
\begin{center}
\vskip -1ex
\rotatebox{0}{\resizebox{!}{0.4\linewidth}{%
  \includegraphics[width=0.8\textwidth]{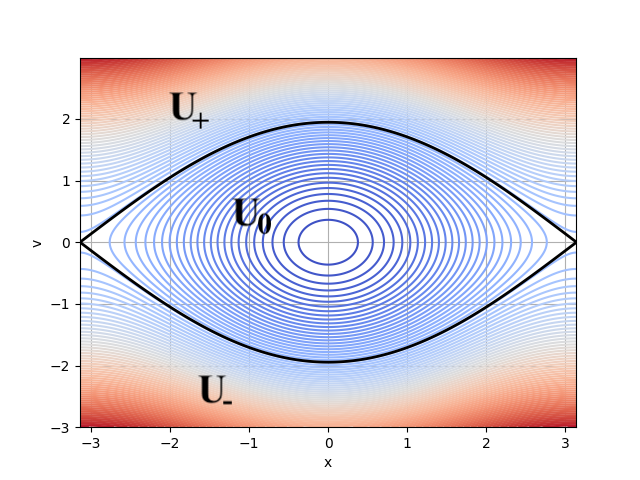}}}
   \caption{Phase portrait of the pendulum Hamiltonian}
   \end{center}
\end{figure}
The transformation $(x, v) \mapsto (\theta, a)$ is symplectic and preserves the Hamiltonian structure and Poisson bracket. It is defined differently in the domains $U_\circ$ and $U_{\pm}$ depicted in Figure \ref{fig1} and deliminated by the two {\em separatrix} lines corresponding to the specific value $h_{M} = M$  of the Hamiltonian. The point $a = 0$ which corresponds to  $h_{M}= -M$  is   the center of $U_\circ$ (the ``eye" of the pendulum). 
Explicit formulae for the transformation are classical and recalled and used  in \cite{Faou2021-qm}. In these coordinates, the action $a$ is constant on each trajectory and the flow of the Hamiltonian $h_{M}$ is linear in the angle: $\varphi^{t}_M(\theta, a) = \left( \theta + \omega(a) t, a \right),$ where $\omega(a)$ is the frequency function on which the motion in $\theta$ depends.  A crucial property of $\omega$  which allows for mixing properties is that $\omega'$ does 
not vanish. 
In  \cite[Theorem 2.2]{Faou2021-qm}, we have quantified  the mixing  effect of this hamiltonian  flow. If $f(t, x,v)$
is the solution of \eqref{freependule} and $\psi(x,v)$ is a smooth test function, we can identify the limit of the correlations, for $t \geq 1$, 
\begin{equation}
\label{dampend}
\left|\int_{\T \times \R} f(t,x,v) \psi(x,v) \dd x \dd v - \sum_{* \in \{ \pm,\circ\}} \int_{J_*} f_0^*(a) \psi_0^*(a)  \dd a \right|  \leq {\frac{C(f^{0},\psi)}{\langle t\rangle^{\frac{p[f^0]+p[\psi]}{2} + 2} }}.
\end{equation}
where we decompose in each domain $U_{*}$ for $* \in \{\circ,\pm\}$ and express 
$$f^{0}(x,v) = \sum_{\ell} f_\ell^*(a) e^{i \ell \theta}, \quad \psi(x,v)= \sum_{\ell} \psi_\ell^*(a) e^{i \ell \theta} $$ for $a \in J_*$ an interval depending on the domain $U_{\ast}$ and $\theta \in [0,2\pi]$. Here, and in the sequel, we use the notation $\langle t \rangle^2 = 1 + t^2$. 
The right hand side  $C(f^0, \psi) $ depends on a certain number of  derivatives of $f^{0}$ and $\varphi$, (typically, we need at least $N \geq \frac{p[f^0]+p[\psi]}{2} + 2$ derivatives to obtain the decay estimates)  and the order of vanishing of a function $g$ is defined by 
\begin{equation}
\label{pq}
p[g]= \max \{n \geq 1  ,  \, \partial^\alpha_{x,v} g(0,0) = 0, \, \forall \alpha, \, 1 \leq |\alpha | \leq n\}\
\end{equation}
with the convention that it is zero if the set in the right hand side is empty. This
measures  the ``flatness" of  a function at  $(0,0)$ (the center  or  the ``eye" of the pendulum). The result  \eqref{dampend} can be understood by the fact that close to $(0,0)$, the action-angle variable degenerates to polar type coordinate  $(x,v) = \sqrt{a} (\cos \theta,\sin\theta)$ and thus results in the emergence of a singularity in high derivatives near $a = 0$.
Note that this is in stark  contrast with the free transport on $\mathbb{T} \times \mathbb{R}$ for which the decay rate of the correlations
is determined only by the regularity, in the analytic case, the decay is exponential.
%
%

We have further  obtained in  \cite{Faou2021-qm} that under  a Penrose stability condition, that we shall recall in a few lines, we have for smooth enough solutions of 
\eqref{linearizedintro} such that the orthogonality condition 
$$ \sum_{* \in \{ \pm,\circ\}} \int_{J_*} g_0^*(a) C_0^*(a) \, da=0$$
holds, 
 the decay of the perturbed magnetization  $\mathcal{C}[g](t)$ at the rate $1/\langle t \rangle^{\max(3, {p[f^0]+5 \over 2}) }$
 and  as a consequence, an $L^1$ scattering result in the sense that  $g(t,\varphi_M^t(x,v))$ has a limit in $L^1_{x,v}$ when $t$ tends to infinity.
 
 
 The goal of this paper is to investigate the nonlinear large time dynamics close to inhomogeneous stationary states.
 We shall focus on particles trapped inside the separatrix where the situation is the most different from the free transport case.
 We shall thus consider a class of equilibria satisfying the following properties:
 
 \begin{Hypothesis}
 \label{hypeq}
 There exists  $M_0>0$,  an open interval $I_0$ containing $M_0$ and $\overline{\kappa}_0 \in (0, M_0/10)$ such  that:
 \begin{itemize}
 \item[{\bf (H1)}]   For every $M\in I_0$, there exists a  stationary solution $\eta_{M}(x,v) = G_M(h_M(x,v))$ such that
  $\mbox{Supp }(G_M) \subset [- M, M -\overline{\kappa}_0 ]$,  $\mbox{Supp }(G_M') \subset [- M+ \overline{\kappa}_0 , M -\overline{\kappa}_0 ]$;
 \item[{\bf (H2)}]  The map $(M, h)\in I_0 \times \mathbb{R} \mapsto  G_M(h)$ is smooth;
  \item[{\bf (H3)}]  The stationary state $\eta_{M}$ satisfies the Penrose stability criterion   uniformly for $M  \in I_0 $:
  \begin{equation}
  \label{condpen} 
  \inf\limits_{\Im \xi \leq  0, \,M \in I_0} \lvert 1 - \hat{K}_{\Cc, M}(\xi) \rvert \geq c_0 > 0, 
  \end{equation}
  where 
$$ K_{\mathcal{C}, M}(t) = \displaystyle\mathds{1}_{\{t\geq 0\}} \int_{ \mathbb{T} \times \mathbb{R} } \cos(X) \bigl\{ \eta_M , \cos(X \circ \varphi^{-t}_M)\bigr\}_{x, v} dx \, dv, 
$$
   \end{itemize}
 
 \end{Hypothesis}
 
 We thus consider a  stationary state which  is compactly supported inside the separatrix,  flat close to the origin and stable and we assume moreover
 that it belongs to a smooth  branch of stationary states  having similar properties in a  vicinity.

 Our first main result is the existence of such states:

\begin{theorem}[Compactly supported stationary state] \label{contruction-stat-state}
There exists a branch of  smooth stationary state  verifying Assumption \ref{hypeq}.
Moreover,  they  are normalized so that  
\begin{equation}
\label{normaliz}
\int_{\mathbb{T} \times \mathbb{R}} \eta_M (x,v) \, dx \, dv = 1, \quad \forall  M \in I_0.
\end{equation}  \end{theorem}

Note that the stationary states that we consider and   that we shall construct here are not small  since $\int_{\mathbb{T} \times \mathbb{R}} \eta_M=1$
and not close to an homogeneous state since in the proof $M_0$ will be taken  close to $1/4$.
 The fact that
these compactly supported equilibria  which are not small satisfy the Penrose stability condition \eqref{condpen} on the torus $\mathbb{T}$ is in stark  contrast
with the case of homogeneous states. As observed in \cite{Toanlin}, for  compactly supported homogeneous states on the torus, 
 the Penrose condition does not hold except if the survival threshold defined there is strictly smaller than  $1$.
 Here we shall consider equilibria such that  though compactly supported inside the separatrix, they satisfy 
   $G_{M}' <0$ sufficiently close to the separatrix 
 and then  make use of the dynamics of the particles 
 close to the separatrix.



Our second main result deals with  the nonlinear system \eqref{eq:VlasovHMFmodel}
with initial condition $f^0$ close to a stationary state given by Assumption \ref{hypeq}  and thus in particular the one constructed in our first result: 
\begin{theorem}
\label{thnonlin}
For   $\sigma \geq 6$,  consider  a family of steady states $\eta_M$ given by Assumption  \ref{hypeq} and an initial datum
\begin{equation}
\label{initicond}
f^0(x,v) = \eta_{M_0}(x,v) + \eps r^0(x,v),  \, \quad \Norm{ r^0}{H^\sigma(\mathbb{T} \times \mathbb{R})} \leq 1 
\end{equation}
such that $r^0$ is even,   compactly supported in  $\{(x,v), \, h_{M_0}(x,v) \in [- M_0 + \overline{\kappa}_0, M_0 - \overline{\kappa}_0]\}$  
with $\overline{\kappa}_0$ given by Assumption \ref{hypeq},  (H1).

Then, there exists $\Lambda >0$, $ \eps_0 >0$, $C>0$ and $\kappa_0 \in (0, \overline{\kappa}_0/4]$   such that for every $\eps \in (0,\eps_0)$, the unique global solution $f(t)$  of \eqref{eq:VlasovHMFmodel}
with initial datum $f^0$ can  be written  for $t \leq { 1 \over C \eps}$ as 
$$ f(t)  = \eta_{M(t)} + \eps r(t) $$
where the modulated magnetization $M(t)$  enjoys the estimate
$$|\dot M(t) | \leq \Lambda  { \eps^2 \over \langle t \rangle^{2\sigma - 5}}, \quad 0 \leq t \leq { 1 \over C \eps}.$$
For  the magnetization  of the perturbation, $\zeta(t) = \mathcal{C}[r](t) = { 1 \over \eps} ( \mathcal{C}[f](t) - M(t) )$,  we have  the estimate 
$$ |\zeta(t) | \leq { \Lambda \over \langle t \rangle^{ \sigma -1}}, \quad 0 \leq t \leq { 1 \over C \eps}. $$
Moreover, the solution $f$ is such that $f - \eta_{M_0}$ remains compactly supported in 
$\{(x,v), \, h_{M_0}(x,v) \in [- M_0 +  \kappa_0, M_0 - \kappa_0]\}$. 
This allows to define a twisted perturbation 
$$ g(t,\theta, a) = r( t, \Phi_{M(t)}( \theta + \int_0^t \omega_{M(\tau)}(a) \, \dd \tau, a))$$
compactly supported in  $\mathbb{T} \times  J$  for some fixed interval $J$ 
where $\Phi_{M}$ is the  action angle change of variable   associated with the $h_M$ Hamiltonian inside the separatrix and 
$\omega_M(a) = {\partial \over \partial a} h_M$. For this twisted perturbation, we have the estimate
$$ {\Norm{g(t)}{H^{\sigma}}\over \langle t  \rangle^3} + \Norm{g(t)}{H^{\sigma-4}} \leq \Lambda, \quad   0 \leq t \leq { 1 \over C \eps}.$$

\end{theorem}

We shall  give more explanations and definitions related to the action angle change of variable and the modulated action angle change of variable in the next section.

Note that our result is valid for large $1/\eps$ but finite times.  This is comparable to the result of  \cite{ChaLuk24} obtained for the Vlasov-Poisson equation
close to zero with an external Kepler potential.  Our analysis is carried out in finite regularity, it is not clear   that analyticity could help to get control over larger  times the main reason is that we loose the convolution structure of the nonlinearity in action angle coordinates so that echoes do not propagate to high frequencies. 
To get a description of the dynamics in times of order  $1/\eps$ in  $H^\sigma$  we use more than just the Lyapounov  stability of the stationary state.
Indeed, an analysis relying only  on that allows to get a control over times of order  $1/\eps^{ 1 \over \sigma+ 1 }$ in $H^\sigma$, here we make in  addition  a crucial use of the mixing properties of the underlying flow.
We also point out that without modulating the magnetization and choosing the appropriate modulated action-angle coordinates, we can reach only
 $1 /\eps^{1\over 2}$.  
 
 One can easily deduce a  modified finite time  scattering result from the above decay estimates since one can find  $g_\infty(\theta, a)$ such that
 $$  \Norm{ g(t) - g_\infty }{L^1} \leq { C\over \langle t\rangle^{\sigma -3} }, \quad t \in [0, C/\eps].$$
 This follows directly from the equation \eqref{eqg} solved by  $g$ and the decay of $\zeta$ and $\dot M$.

\section{Action-Angle coordinates}

In this section, we recall some standard properties of the integrability of the pendulum Hamiltonian $h_{M}= { v^2/2}- M \cos x$.
The main properties useful for the study of linear Landau damping have been carefully studied  in \cite{Faou2021-qm}.
Nevertheless, here, since we will need to modulate the magnetization parameter of the stationary state, we have to carefully track 
the $M$ dependence.

The phase space \( \T \times \R \) is split into the three regions: for $h_M(x, v)$ the Hamiltonian given by \eqref{eqh0} we have
$$
\left|
\begin{array}{l}
U_{\circ} =  \left\{ \left( x, v \right) \in \T \times \R \mid - M < h_M(x, v) < M \right\},\\
U_{+} =  \left\{ \left( x, v \right) \in \T \times \R \mid h_M(x, v) > M, \, v > 0 \right\}, \\
U_{-} =  \left\{ \left( x, v \right) \in \T \times \R \mid h_M(x, v) > M, \, v < 0 \right\},
\end{array}
\right.
$$
 with the set \( \{(x, v) \in \mathbb{T} \times \mathbb{R} \mid h_{M}(x, v) = M \} \) called the separatix. The transition from physical coordinates \( (x, v) \in U_* \) to action-angle coordinates \( ( \theta, a) \in \mathbb{T} \times J_{*}\), $* \in \{ \circ,\pm\}$ is established using  first mixed-variable generating functions and then the transition from  $(h, \theta)$ to $(a, \theta)$  is explicitly calculated, see for example  \cite[Proposition~7.5]{Faou2021-qm} and \cite[Proposition~7.9]{Faou2021-qm}, they involve Jacobi elliptic functions. Here, we are mainly interested in $U_{\circ}$. To stress the dependence in $M$, we shall rather use the notation $U_{M, \circ}$, the corresponding interval for actions is    $J_{M, \circ} = \left( 0, \frac{8}{\pi} \sqrt{ M} \right)$ and sometimes we shall drop the $\circ$  to denote the actions interval corresponding to inside the separatrix. We then  use the following notation 
 for the change of variable:
  $$ \Phi_{M}:\quad  ( \theta, a) \in \mathbb{T} \times J_{M, \circ} \mapsto  \left(x=x_{M}(\theta, a),  v=v_{M}(\theta, a)\right) \in  U_{M,\circ}$$
Note that $\Phi_{M}$  has   analytic regularity.

 Let us just recall for $k \in (0,1)$ and $\phi \in (-\pi/2,\pi/2)$ the definition of  the incomplete elliptic integrals
\begin{equation}
\label{ellEF}
E(\phi,k) = \int_0^\phi \sqrt{1 - k^2 \sin(y)} \dd y \quad \mbox{and} \quad F(\phi,k) = \int_0^\phi \frac{1}{\sqrt{1 - k^2 \sin(y)}} \dd y
\end{equation}
and the complete elliptic integrals by
\begin{equation}
\label{ellEF2}
\bE(k) = E\left(\frac{\pi}{2},k\right) \quad \mbox{and} \quad \bK(k) = F\left(\frac{\pi}{2},k\right).
\end{equation}
We have  for  the map $a=a_{M}(h)$,  which allows to go from the $h$ variable
 to the action variable,  that
$a'_{M}(h)= { 1 \over \omega_{M}(h)}>0$,  for $ h \in (-M,M)$ and 
\begin{equation}
\label{bigom}
\omega_{M}(h) =\frac{\pi\sqrt{M}}{2 \bK(k_M(h))}, \quad a_{M}(h)= \frac{ 8 \sqrt{M}}{\pi}\Big[
 \bE(k_M(h)) - ( 1- k_M(h)^2)\bK(k_M(h))\Big]
\end{equation}
with $k_M(h) = \sqrt{ \frac{h + M}{2 M} }$. 

We shall denote by $H_{M}(a)$ the inverse of $a_{M}$ so that 
\begin{equation}
a_M[H_M(a)] = a,\quad \mbox{and} \quad 
\label{omegaMdef} H_{M}'(a)= \omega_{M}(a).
\end{equation}

A crucial property is that
\begin{equation}
\label{omega'}
\omega_M'(a) <0 , \quad a \in  J_{M, \circ}.
\end{equation}

The Fourier expansion in $\theta$ of  $\cos x_{M}(\theta, a)$ will be important in the analysis, we set
\begin{equation}
\label{cosexpansion}
 \cos x_{M}(\theta, a) = \sum_{\ell \in \mathbb{Z}} C_{ \ell,M}^\circ(a) e^{ i \ell \theta}.
 \end{equation}
For $\ell \in \Z,$  the expression in 
 $U_{M, \circ}$ of  the Fourier coefficients $C_{\ell,M}^\circ(a)$   (and also the related expansions outside the separatrix)
 are explicitly calculated from Jacobi Elliptic Functions as analytic functions 
  of the energy.
 We have
\begin{align}
 \label{coeffimpair} & C_{2\ell+1, M}^{\circ}(h)   = 0, \\
 \label{coeff0}&C_{0, M}^{\circ}(h) = -1 + 2 \frac{\boldsymbol{E}(k_M(h))}{\boldsymbol{K}(k_M(h))} \\
 \label{coeffpaszero} & C_{2 \ell, M}^{\circ}(h) = C_{- 2 \ell, M}^{\circ}(h)  =(-1)^{|\ell|} \frac{2\pi^{2}}{\boldsymbol{K}(k_M(h))^{2}} \Bigl( \frac{|\ell|\,q(k_M(h))^{|\ell|}}{1 - q(k_M(h))^{2|\ell|}} \Bigr), \quad \ell \neq 0 \end{align}
for  \(k_M(h) = \sqrt{\frac{h+M}{2M}}\) and
Jacobi's nome $q(k) = \exp( -\pi \bK(\sqrt{1 - k^2}) / \bK(k))$.
  We again refer to   \cite[Proposition~7.1]{Faou2021-qm} and references therein.

A useful  property of these coefficients that we see on these expressions  is that for $\ell \neq 0$, $C_{2 \ell, M}^{\circ}$ does not vanish.

We shall omit the subscript $M=1$ when we deal with the quantities associated with the Hamiltonian $h(x,v):=h_{1}(x,v)= v^2/2 - \cos x.$
Since $h_{M}(x,v)= M h(x, v/\sqrt{M})$, we have 
 $\varphi_{M}^t(x,v) = \varphi_{1}^{\sqrt{M} t}(x,\frac{v}{\sqrt{M}})$ and we can relate $\Phi_{M}$ and $\Phi$ in the following way: As action-angle changes of variables are symplectic this diagonal linear change of variable is carried over in variable $(\theta,a)$. 
 Hence we have 
\begin{equation}
\label{scalingPhiM} x_{M}(\theta, a)= x(\theta,  {a\over \sqrt{M}}), \quad  \frac{v_{M}(\theta, a)}{\sqrt{M}}=  v(\theta, { a \over \sqrt{M}}).
\end{equation}
Furthermore for, $a \in J_{M,\circ}$, we also  have
\begin{equation}
\label{hMomegaM}
H_{M}(a)= M H\left( { a \over  \sqrt{M}}\right), \quad \omega_{M}(a)= \sqrt{M} \omega\left({ a  \over \sqrt{M}}\right), 
\end{equation} 
which can be seen directly  from \eqref{bigom}, as $k_M( hM) = k_1(h)$ and as $H_M$ is the inverse of $a_M$ satisfying \eqref{omegaMdef}. 
For some $M_{0}>0$, given by  Assumption \ref{hypeq}, 
we shall be interested in functions which are compactly supported in 
\begin{equation}
\label{defKM0}
\mathcal{K}_{M_0}=\{ (x,v), \,  - M_{0}+ \kappa_{0} \leq  h_{M_{0}}(x,v) \leq M_{0} - \kappa_{0}\}
\end{equation}
for some $\kappa_{0} >0$ slightly larger than $\overline{\kappa}_0$ provided by Assumption \ref{hypeq}.
For  $M \in [M_{0}- \mu , M_0+ \mu]$, we observe that for $\mu>0$ sufficiently small, we have
\begin{equation}
\label{K0inclus}
 \mathcal{K}_{M_0} \subset \{ (x,v), \,  - M+ { 3\kappa_{0}  \over 4}\leq  h_M \leq M - {3 \kappa_{0} \over 4} \}, 
 \end{equation}
so that $\mathcal{K}_{M_0}$ can be smoothly  parametrized by using  any of the action angle coordinates $\Phi_{M}$.

In order to use modulation theory with respect to $M$, we shall use the following:
\begin{lemma}
\label{lemMaction}
For  $\kappa_{0} \in (0, M_{0})$, and $\mathcal{K}_{M_0}$ defined by \eqref{defKM0}. Then 
there exists $\mu>0$ and an interval  $[a_{-}, a_{+}]$ such that the map:
 $$ (M, \theta, a) \in [M_{0}- \mu, M_{0}+ \mu] \times \mathbb{T} \times [a_{-}, a_{+}] \mapsto \Phi_{M}(\theta, a)$$
 is analytic. Moreover, for every $M \in [M_{0}- \mu, M_{0}+ \mu]$, the map $\Phi_{M}: \mathbb{T} \times [a_{-}, a_{+}]
  \mapsto U_{M, \circ}$ is a smooth diffeomorphism onto its image which contains $\mathcal{K}_{M_0}$
  and is included  in  $$ \widetilde{\mathcal{K}}_{M_0}=\{ (x,v), \,  - M_{0}+ {\kappa_{0}\over 2} \leq  h_{M_{0}}(x,v) \leq M_{0} - {\kappa_{0} \over 2}\}.$$
\end{lemma}
\begin{proof}
Let us set
$$ a_{+}= \sqrt{M_{0}} H_{1}^{-1 } \left( { M_{0}} - { \kappa_{0}  \over 2} \over M_{0} \right), \quad  a_{-}= \sqrt{M_{0}} H_{1}^{-1 } \left( {-  M_{0}} + { \kappa_{0} \over 2} \over M_{0} \right)$$
so that $H_{M_{0}}( a_{\pm}) = \pm M_{0} \mp { \kappa_{0} \over 2}.$ We then observe that we can find $\mu$  sufficiently small
 so that for every $M \in [M_{0}- \mu, M_{0}+\mu]$, we have
\begin{align*}
\mathcal{K}_{M_0} &\subset \{ (x,v), \,  h_{M}(x,v) \in  H_{M}[a_{-}, a_{+}] = [ H_{M}(a_{-}), H_{M}(a_{+})]\} \\&\subset \left\{ (x,v), \,  - M+ {\kappa_{0}  \over 4}\leq  h_M \leq M - {\kappa_{0} \over 4} \right\}.
\end{align*}
  Indeed, let us define for $M$  close to $M_{0}$, 
  $$ \varphi_{\pm}(M)= H_{M}(a_{\pm})=  M H\left( {a_{\pm} \over \sqrt M}\right).$$
  Then we have  
  $ \varphi_{\pm}(M_{0})= \pm M_{0}\mp { \kappa_{0} \over 2}$  and $\varphi$ is a smooth function for $M$ close to $M_{0}$.
  One can thus take $\mu>0$  sufficiently small to that $\varphi_{+}(M)  \geq  M - { 3 \kappa_{0} \over 4}$ and $\varphi_{+}(M) \leq M- { \kappa_{0}\over 4}$ for $M \in [M_{0}- \mu, M_{0} + \mu]$
  and in a similar way $\varphi_{-}(M) \leq - M + { 3 \kappa_{0} \over 4}$ and $\varphi_{-}(M) \geq - M + { \kappa_{0} \over 4}.$
  We can thus consider the map
  $$ (M, \theta, a) \in [M_{0}- \mu, M_{0}+ \mu] \times \mathbb{T} \times [a_{-}, a_{+}] \mapsto \Phi_{M}(\theta, a).$$
  From the scaling property \eqref{hMomegaM}, it is analytic and by construction, we have that  for every $M \in  [M_{0}- \mu, M_{0}+ \mu]$, 
    $\Phi_{M}$ is a diffeomorphism  from $ \mathbb{T} \times [a_{-}, a_{+}]$ to its image which  contains $\mathcal{K}_{M_0}$ by using \eqref{K0inclus}
    and is contained strictly inside the separatrix. 
  \end{proof}

An important consequence of Lemma \ref{lemMaction} is the following:
 \begin{Remark}\label{action-angle-equivalence}
For a function $f(x,v)$ supported in $\mathcal{K}_{M_0}$ and $M \in [M_0-\mu, M_0+\mu]$ , we can define a function $g$ defined in $\mathbb{T} \times [a_-, a_+]$  by
\begin{equation}
\label{defequiv}
g(\theta, a) = f \circ \Phi_M(\theta, a),
\end{equation} 
and by using Fa\`a di Bruno formula and the previous lemma, we  get  that $g$ and $f$ have the same regularity.
In particular, $f$ is in $H^s$ if and only if $g$ is in $H^s$. Moreover, we have the   norm equivalence
\begin{equation}
\label{equivnorme}
\Norm{f}{H^s_{x,v}} \simeq  \Norm{g}{H^s_{\theta,a}}  \simeq \left(  \sum_{0 \leq p \leq s}\sum_{\ell \in \Z} \langle \ell \rangle^{2p} \int |\partial_a^{s-p} \tildeg_\ell (a)|^2\right)^\frac12
\end{equation}
where we fit in the Fourier-\(\theta\) expansion $ \tildeg(a, \theta) = \sum\limits_{\ell \in \mathbb{Z}} {\tildeg}_{\ell}(a) e^{i \ell \theta}.$ 
\end{Remark}

Finally, note that by  using the scaling property \eqref{scalingPhiM}, we can write
\begin{equation}
\label{scaledcos}
 \cos x_{M}(\theta, a) =   \sum_{\ell \in \mathbb{Z}} C_{\ell,M}\left(a\right) e^{ i \ell \theta}=\sum_{\ell \in \mathbb{Z}} C_{\ell}\left({a\over\sqrt{M}}\right) e^{ i \ell \theta}
\end{equation}
and thus recover  that the map $M \in [M_{0}- \mu, M_{0}+ \mu] \times [a_{-}, a_{+}] \mapsto C_{\ell,M}(a)$  is analytic.
  
  \section{A useful technical Lemma.}
  Before starting the proof of our main results, we shall state a useful version of the classical non-stationary phase Lemma 
  which allows us  to get a cheap version of \eqref{dampend} for functions compactly supported in $\mathcal{K}_{M_0}$.
This will be  instrumental for  the analysis  of  the decay of the magnetization.
\begin{Lemma}\label{time-decay-ipp}
For every  $k \in \N$ and every interval $[a, b] \subset  J_{ M, \circ} = \left( 0, \frac{8}{\pi} \sqrt{ M} \right)$ for every $M \in [M_0 - \mu, M_0 + \mu]$, there exists $C>0$, such that the following holds: 
for every  $ \mathcal{A} \in H^{k}(\R)$ with compact support in   $[a, b]$, 
 $$ 
\big\lvert I[\mathcal{A}, \omega_M](\xi) \big\rvert \leq \frac{C}{ \langle \xi\rangle^{ k } } \Norm{\mathcal{A}}{ H^{k } },  \quad \forall \xi \in \mathbb{R}, 
$$
where 
\begin{equation}
\label{IAOS}
I[\mathcal{A}, \omega_M](\xi) := \int  \mathcal{A}(a)  e^{i \xi \omega_M(a)} \dd a.
\end{equation}
\end{Lemma}

\begin{proof}
For $k=0$, the estimate just follows from Cauchy-Schwarz so that  only  $k\geq 1$ and $|\xi |\geq 1$ are interesting to study.
The crucial property is \eqref{omega'} which ensures that the phase is non-stationary.
We can thus write 
$$ 
e^{i \xi \omega(a)} = \frac{1}{i \xi \omega_M'(a)} \frac{d}{da} \left[  e^{i \xi \omega(a)} \right].
$$
and integrate by parts $k$ times after replacing in the definition of  $I[\tildeg, \Theta]$. yields
\begin{align}
I[\mathcal{A}, \omega_M](\xi) &= \sum\limits_{\ell = 1}^{k } \frac{(-1)^\ell}{\left( i \xi\right)^{\ell} } \left[ \mathcal{D}^{\ell -1}( \mathcal{A} )  e^{i  \xi \omega_M(a)}\right]_{a}^b +  \frac{ (-1)^{k +1 } }{ \left( i \xi \right)^{k } } \int  \mathcal{D}^{k} ( \mathcal{A} ) e^{i \xi \omega(a)}\dd a. \label{ipp-Ig}
\end{align}
where
$$ \mathcal{D}(f) :=  \frac{d}{da} \left[ { f(a) \over  \omega_M'(a)}  \right].$$ 
Noticing that there is no boundary term because of the support assumption, the result easily follows thanks to \eqref{omega'} and the smoothness of $\omega_M$ on the support of $\mathcal{A}$.
\end{proof}
As an easy consequence of the above Lemma, we get 
the following  version of 
 \cite[Theorem 2.2]{Faou2021-qm} for compactly supported functions avoiding the origin. 
\begin{Corollary}\label{corchiant}
Let $k \in \N^*$, there exists $C_1>0$ and $C_2>0$  such that 
for every  $f  \in H^{k}(\T \times \R)$ with compact support in $\mathcal{K}_{M_0}$, every  $ \phi \in W^{k, \infty}( \mathcal{K}_{M_0} ) $
and every $M \in [M_0-\mu, M_0+\mu],$  we have decomposing in action angle 
$$g= f \circ \Phi_M(\theta,a) = \sum_{\ell \in \Z} g_{\ell}(a) e^{i \ell \theta}, \quad \psi(\theta,a) = 
 \phi \circ \Phi_M= \sum_{\ell \in \Z} \psi_{\ell}(a) e^{i \ell \theta}$$  
 and using  $\varphi^t_M$ the flow of the Hamiltonian $h_M$,  that
\begin{align*}
\left|
\int_{\mathcal{K}_{M_0}} f(x,v)\psi(\varphi^{-t}_M( x,v))dxdv-
\int_{J_M } \tildeg_0(a) \psi_0(a) da \right|&\leq \frac{C_1}{\langle t \rangle^k} \Norm{g }{H^k(\mathbb{T} \times [a_-, a_+])} \Norm{\psi }{W^{k, \infty}([a_-, a_+] ) }
 \\  &\leq  \frac{C_2}{\langle t \rangle^k}  \Norm{f }{H^k} \Norm{\psi } {W^{k, \infty}}. 
\end{align*}
\end{Corollary}
\begin{proof}
It is a consequence of the fact that 
$$
\int_{\T\times J_{\circ}} f (x,v)\psi(\varphi^{-t}_M( x,v))dxdv = \sum_{\ell \in \Z} \int_{J_\circ} \tildeg_\ell(a) \psi_{-\ell}(a) e^{-i t \ell  \omega(a)}da = \sum_{\ell \in \Z}  I[\tildeg_\ell \psi_{\ell}, \omega_M](t). 
$$
And thus by the previous Lemma and the Cauchy-Schwartz inequality, we have
$$
\left| \sum_{\ell \neq 0}  I[\tildeg_\ell \psi_{\ell}, \omega_M](t) \right| \leq \frac{C}{\langle t \rangle^k} \sum_{\ell \in \Z} \Norm{\tildeg_\ell}{H^k_a} \Norm{\psi_\ell}{H^k_a} \leq \frac{C}{\langle t \rangle^k} \Norm{\tildeg}{H^k} \Norm{\psi}{H^k}. 
$$
The final estimate above  follows from Remark  \ref{action-angle-equivalence} and Lemma \ref{lemMaction}.
\end{proof}

\section{Stable stationary state with compact support: proof of Theorem \ref{contruction-stat-state}} 
\label{section-stationary-state}

\subsection{Existence}

As already explained, for any smooth function $G$, we have an even  stationary state under the form
$$
 \eta(x,v) = G(h_M(x, v))$$
for some magnetization $M>0$ if  one can find $M>0$ solution of the equation \eqref{lequationsurM}. Allowing $G$ to depend on $M$, we can make the choice, 
$$ \eta(x,v) = G_M(h_M) = \Psi(k_M(h_M)^2), \quad   k_M(h) = \sqrt{ \frac{h + M}{2 M}}.$$
Note that $G_M$ and $\Psi$ are thus related through
$$ G_M(h) = \Psi \left( {h+ M\over 2M}\right).$$ 
With this choice,  thanks to a change of variable in $v$, we observe that 
\begin{multline*}
\int_{\T \times \R} \cos(x)  G_M(h_M(x,v))\, dx dv = 
 \int_{\T \times \R} \cos(x) \Psi \left[ {1 \over 2 }\left( \frac{v^2}{2M}  + 1 -  \cos(x)\right)\right]   \, dx \, dv \\
= M^{1 \over 2 }  \int_{\T \times \R} \cos(x) \Psi \big[k\big(h(x,v)\big)^2\big]\, dx \, dv, \quad  h(x,v)= { |v|^2 \over 2} - \cos x, \quad  k(h)= \sqrt{{ h+1 \over 2 }}, 
\end{multline*}
so that \eqref{lequationsurM} reduces to 
\begin{equation}
\label{Meqstat} M^{1 \over 2 }  = \int_{\T \times \R} \cos(x) \Psi \big[ k\big(h(x,v)\big)^2\big]   \, dx \, dv.
\end{equation}
One  then easily find a large class of  stationary solutions  by taking any  $\Psi$  such that  
$$\int_{\T \times \R} \cos(x) \Psi \big[ k\big(h(x,v)\big)^2\big]      \, dx \, dv >0$$
and then  by choosing $M$ defined by \eqref{Meqstat}. Nevertheless, this does not allow to solve the normalization condition
\eqref{normaliz} and  does not ensure the stability condition  (H3), moreover, some work is still needed to check that the condition \eqref{Meqstat}
is compatible  with the support properties that we want.
We shall solve these issues by also  allowing $\Psi$ to depend on $M$.

It is first convenient to use action angle coordinates to reformulate \eqref{Meqstat}.

%
%
\begin{Lemma} \label{character-eta-M0}
Let \( \Psi : \mathbb{R} \to \mathbb{R}^{+} \) be a smooth  function compactly supported in $(-\infty, 1)$
. For any given magnetization constant $M$, the function
\begin{equation} \label{eqeta}
\eta   =  \Psi \Big(\frac{v^2}{4M} +  \frac{1 - \cos (x)}{2}\Big) = \Psi \Big( \frac{h_M + M}{2M}\Big) = \Psi(k_M^2), \quad \text{ with } \quad k_M(h) = \sqrt{ \frac{h + M}{2 M} },
\end{equation}
defines a stationary state, with \(M\) satisfying
\begin{equation}
\label{defM0}
\sqrt{M} = \frac{4}{\pi} \int_{0}^{1} P(z) \Psi(z) \, \dd z
\end{equation}
where \(P(z)\) is defined as
\begin{equation}
\label{eqP}
P(z) = 
2 \bE(\sqrt{z})  - \bK(\sqrt{z}), \quad  z \in (0,1). 
\end{equation}
Moreover, we have 
\begin{equation}
\label{eqInt}
\int_{\mathbb{T} \times \mathbb{R}} \eta(x,v) \, \dd x \, \dd v = \frac{4 \sqrt{M}}{\pi} \int_0^1 \bK(\sqrt{z}) \Psi(z)  \dd z.
\end{equation}
\end{Lemma}

\begin{proof}
Since $(x,v) \mapsto  \Psi \big[ k\big(h(x,v)\big)^2\big] $  is supported in $U_\circ$, we use in the right hand side of \eqref{Meqstat}  the changes of variable
\begin{equation}
\label{chgh}
 (x, v) \mapsto (a, \theta) \mapsto (h, \theta), \quad   dx \, dv = da \, d\theta = \frac{1}{\omega(h)} dh \, d\theta, 
 \end{equation}
  thanks to \eqref{omegaMdef} and the expansion \eqref{cosexpansion}, we get 
\begin{equation}
 \label{eq-magnetization-constant}
M^{1\over 2 }= \sum_{\ell \in \mathbb{Z}} \left( \int_{\mathbb{T} \times  (-1,1)} \frac{C^{\circ}_\ell(h)}{\omega(h)} e^{i \ell \theta} \Psi \left( k^2(h) \right) \, \dd h \dd \theta \right)
= \int_{-1}^{1} \frac{C^{\circ}_0(h)}{\omega(h)} \Psi \left( k^2(h) \right) \dd h 
\end{equation}
Note that we have used the fact that non-zero modes \((\ell \neq 0)\) do not contribute  after integration over \(\theta.\) Using the expressions of the Fourier coefficients $C^{\circ}_0(h)$  \eqref{coeff0} and the expression \eqref{bigom} of $\omega(h)$, we have 
$$
\frac{C_0(h)}{\omega(h)} =  
\frac{2}{\pi} (2 \bE(k(h))  - \bK(k(h))) \quad \mbox{for} \quad h \in (-1,1)
.$$ 
Substituting the latter into \eqref{eq-magnetization-constant} and changing variables to \( z = k^2=  \frac{h + 1}{ 2  }  \) 
$$
\sqrt{M} = \frac{8}{\pi}  \int_0^1 k \, (2 \bE(z^{ 1\over 2} )  - \bK(z^{ 1 \over 2}))  \Psi(z)  \dd z 
$$
We have thus obtained  \eqref{defM0}. 
The integral relation \eqref{eqInt} for \(\eta(x, v)\) follows from \[ \int_{\T \times \R} \eta(x, v) \, dx \, dv =   \int_{-M}^{M} \frac{1}{\omega_M(h)} \Psi \left( k_M(h)^2 \right) \dd h  \]
 and  we proceed exactly as before to reach the required conclusion.
\end{proof}


We will construct a familly of stationary states  as above by introducing an auxiliary scaling parameter $\delta>0$, we set $\eta_\delta (x,v)= G_\e(h_{M_\e}) = \Psi_\delta( \frac{h_{M_\delta} + M_\delta}{2M_\e})$, where  \( \Psi_\delta \) is  the sum of two smooth scaled  functions with scaling laws depending on $\delta$.  Roughly, the stationary state will  look like a small  amplitude plateau  bump function
which has  large support inside the separatrix, it will be  almost touching it,    with  another plateau  function  with large amplitude  and small support  around zero on top of it.

\begin{lemma}
\label{lemmmzzer}
Consider $\phi \in \mathcal{C}^{\infty}_c \left( \R, \R^{+} \right) $ an even  and non-increasing  function  on $\mathbb{R}_+$ with compact support in \( (-1, 1)\) such that 
$\phi'$ has compact support in $(-1,1)\backslash\{0\}$ and we impose 
\[ 
\int_{0}^1 \phi(z) \dd z = 1.
\] 
  Consider also a non-increasing  function $\chi \in \mathcal{C}^{\infty}_c  \left( \mathbb{R}, \mathbb{R}^{+} \right)$  such that \( 0 \leq \chi(z) \leq 1 \text{ for } z \in \mathbb{R}\) and verifying 
\begin{equation}
\label{propchi}
\chi(z) = 1, \quad  \mbox{for} \quad   z < 0,  \quad \chi(z) = 0, \quad \mbox{for} \quad z > 1, \quad \chi'<0, \quad \mbox{for}\quad z\in (0, 1)
\end{equation}
Define the family of functions $\Psi_{\delta} \color{black} $ by the relation 
\begin{equation}
\label{dqsdoi} \Psi_\delta(z) = \frac{\alpha}{\e}\phi\big(\frac{z}{\e}\big) + \beta \chi\big(\frac{z - \nu_{1}(\e)}{\nu_{2}(\e) - \nu_{1}(\e)}\big) . 
\end{equation} 
For \(0 < \e \ll 1,\) we fix the following parameter laws in terms of $\e$: 
\begin{equation}
\label{aedmoqdij}
\nu_{1}(\e) = 1 - \e, \quad \beta = \e^{2} \quad \mbox{and} \quad \nu_{2}(\e) = 1 - e^{- \e^{-20}}. 
\end{equation}
Then there exists $\e_*> 0$ such that for all $\e \in (0, \e_* )$, there exist 
\begin{equation}
\label{alphamzero}
\alpha = \alpha_{\e} = \frac{1}{4} + \frac{m_{1} }{8} \e + \mathcal{O}(\e^{2})  \quad \mbox{and} \quad M_\delta = \frac{1}{4}  - \frac{ m_{1} }{8} \e + \mathcal{O}(\e^{2}), 
\end{equation}
where 
\begin{equation}
\label{m1def}
m_{1} := \int_{0}^{1} z \phi(z) dz >0
\end{equation}
such that $\Psi_{\delta}$ fulfills the conditions of the Lemma \ref{character-eta-M0} for $\eta_\e = \Psi_{\delta}(k_{M_\delta}^2(h_{M_\delta}))$ to define a stationary state of the Vlasov-HMF equation.

Moreover, the map  $\delta\in [0, \delta_*]  \mapsto M_\e \in [M_-,  1/4]$ for some $M_-\in (0, 1/4)$  is a smooth strictly decreasing  diffeomophism. 
\end{lemma}
Let us make the following comments: 

\noindent
(i) Note that the hypothesis on $\phi$  implies that $\phi$ is constant near $z = 0$  so that  $\Psi_\delta$  is also constant close to zero,  moreover  $\Psi_\delta \equiv 0$ near $z = 1$,  and  $\Psi_\delta$ is non-increasing on $\mathbb{R}_+$.

\noindent
(ii)  Since the map $\delta\mapsto M_\delta$ is invertible one can take $M$ as a parameter instead of $\delta,$ by denoting $M\mapsto\delta(M)$
 the inverse, 
we thus get a stationary state
\begin{equation}
\label{defetaMpreuve}
\eta_M(x,v)=  \Psi_{\delta(M)}(k_M(h_M)^2)
\end{equation}
defined for $M \in [M_-,  1/4)$.

\noindent
 (iii)  Owing to the fact that $\frac{1}{\nu_{2}(\e) - \nu_{1}(\e)}  = \frac{1}{\e - e^{- \e^{-20}}}\leq \frac{2}{\e}$  and that for all $k$, $\partial_\e^k ( e^{- \delta^{-20}}) \leq 1$ for $\e$ small enough, we easily obtain that $\Psi_\delta$ is smooth in $\delta$   and that  for all $k$, $\ell$,  there exists $C_{k,\ell}$ such that for all $\e > 0$, 
\begin{equation}
\label{deriv1}
\Norm{\partial_\delta^{k} \partial_z^\ell\Psi_\delta(z)}{L^\infty} \leq\frac{C_{k,\ell}}{\e^{k+\ell+1}} 
\end{equation}
We thus have that the steady state defined by \eqref{defetaMpreuve} fulfills for any $M_0 \in (M_-,  1/4)$ the properties   (H1) and (H2)  of Assumption \ref{hypeq} by taking $I_0$ sufficiently small containing $M_0$.


\begin{proof}
From Lemma \ref{character-eta-M0}, the magnetization condition imposes the relationship
\begin{align*}
\sqrt{M_\e} &= \frac{4 }{\pi} \int_0^\infty P(z) \Psi_{\delta}(z)  \dd z
\end{align*} 
By definition, $\Psi_{\delta}(z) = 0$ for $z > \nu_{2}(\e)$, since for such $z$, we have  \( \frac{z}{\e} > \frac{\nu_{2}(\e)}{\e} > 1 \) such that \(\phi \left( \frac{z}{\e} \right) = 0\) and \( \frac{z - \nu_{1}(\e)}{\nu_{2}(\e) - \nu_{1}(\e)} > \frac{\nu_{2}(\e) - \nu_{1}(\e)}{\nu_{2}(\e) - \nu_{1}(\e)} > 1\) such that \( \chi \left( \frac{z - \nu_{1}(\e)}{\nu_{2}(\e) - \nu_{1}(\e)} \right) = 0.\) Hence, we can decompose the integral on the domain \( z \in \left[ 0, \nu_{2}(\e) \right] \) to have
\begin{align*}
\sqrt{M_\e} = \underbrace{ \frac{4 \alpha_{\e} }{\pi \e } \int_{0}^{\e} \phi \left( \frac{z}{\e} \right) P(z) \dd z }_{I_{1}}
\, +  \, \underbrace{ \frac{4 \beta }{\pi } \int_{0}^{\nu_{1}(\e)}  P(z) \dd z }_{I_{2}} \,
+ \, \underbrace{ \frac{4 \beta }{\pi } \int_{\nu_{1}(\e)}^{\nu_{2}(\e)}  \chi\big(\frac{z - \nu_{1}(\e)}{\nu_{2}(\e) - \nu_{1}(\e)}\big)   P(z) \dd z }_{I_{3}}. 
\end{align*}

From  \cite[Proposition~7.1]{Faou2021-qm}, $P(z)$ only has logarithmic singularity at $z = 1.$ From formulae (900.00) and (900.07) of \cite{byrd1971handbook}, we have the relationship \( P(z) = 2 \boldsymbol{E} \left( \sqrt{z}\right) - \boldsymbol{K} \left( \sqrt{z}\right) = \frac{\pi}{2} - \frac{3 \pi}{8} z + \mathcal{O}(z^{2}) \) for \( z \to 0\). Similarly, we  asymptotic behaviour \( P(z) = \mathcal{O} \left( \log(1 - \sqrt{z}) \right) \) near \(z = 1 \) which can be applied throughout to evaluate the integrals.

\begin{itemize}
    \item \textbf{Term \( I_{1}:\)} Substituting \( z = \e z',\) we have 
\begin{align*}    
     I_{1} &= \frac{4 \alpha_{\e} }{\pi } \int_{0}^{1} \phi(z') P(\e z') \dd z' = \frac{4 \alpha_{\e} }{\pi } \int_{0}^{1} \phi(z') \left( \frac{\pi}{2}   - \frac{ 3 \pi }{ 8 } \e z' + \mathcal{O}(\e^{2} z'^{2}  ) \right) \dd z'\\
     & = \alpha_{\e} \left( 2  - \frac{3 m_{1} }{ 2} \e + \mathcal{O} \left( \e^{2} \right) \right), 
\end{align*}    
since $\int_{0}^{1} \phi(z) dz = 1.$ 
    \item \textbf{Term \( I_{2}:\)}  We have  \( P(z) = \frac{\pi}{2} + \mathcal{O}(z) \) when $z \to 0$ and $P(z) = \mathcal{O}(\log( 1 - \sqrt{z}))$ when $z \to 1$ and thus $P$ is integrable on $(0,1)$. As \(\beta = \e^{2},\) we have upper bound 
\[ 
\left| I_{2} \right| \leq \frac{4\beta}{\pi} \int_0^1 |P(z)| d z = \mathcal{O} \left( \e^{2} \right).
\]
    \item \textbf{Term \( I_{3}:\)} We use \( P(z) \leq C \left| \log(1 - \sqrt{z}) \right|\) near \(z = 1,\) \( \beta = \e^{2}\) and \(\lVert \chi \rVert_{\infty} \leq 1,\) we have with the symbol $A \lesssim B$ to denote the existence of a constant $C$ independent of $\e$ such that $A \leq CB$, 
    \begin{align*} 
    \left| I_{3} \right| &\lesssim \e^2 
    \left( - \int_{\nu_{1}(\e)}^{\nu_{2}(\e)} \log(1 - \sqrt{z}) d z \right) \lesssim \e^2 
    \left( - \int_{\sqrt{\nu_{1}(\e)}}^{\sqrt{\nu_{2}(\e)}} \log(1 - u) u d u \right)\\ 
    & \lesssim \e^2 
    \left( - \int_{\sqrt{\nu_{1}(\e)}}^{\sqrt{\nu_{2}(\e)}} \log(1 - u)  d u \right) \leq \frac{4 C}{ \pi } \beta
    \Big[( 1 - u)  \log(1 - u) - (1 - u) \Big]_{\sqrt{\nu_{1}(\e)}}^{\sqrt{\nu_{2}(\e)}}= \mathcal{O}( \e^3 \log \e). 
    \end{align*} 
\end{itemize}

Hence, we have 
\begin{align}
 \sqrt{M_\delta} = I_{1} + I_{2} + I_{3} = \alpha_{\e} \left( 2 - \frac{3 m_{1}}{2} \e + \mathcal{O}(\e^{2}) \right) + \mathcal{O}(\e^{2}). \label{condition-M-0}
\end{align}
Note that we can similarly use the asymptotic limit of \( \boldsymbol{K}(\sqrt{z}) = \frac{\pi}{2} \left( 1 + \frac{z}{4} + \mathcal{O} \left( z^{2} \right) \right) \)  from formula (900.00) of \cite{byrd1971handbook} to calculate 
\begin{align*}
 \int_{ \mathbb{T} \times \mathbb{R}} \eta_{\e}(x,v) \dd x \, \dd v &= \frac{ 4 \alpha_{\e} \sqrt{ M_\e } }{ \pi } \int_{0}^{1} \phi( z' ) \boldsymbol{K} \left( \sqrt{\e z'} \right) d z' + \mathcal{O}(\e^{2}) \\&= 2 \sqrt{M_\e} \alpha_{\e}\int_{0}^{1} \phi(z') \left( 1 + \frac{ \e z' }{ 4 } + \mathcal{O} \left( \e^{2} \right) \right) dz' + \mathcal{O}(\e^{2})
\end{align*}
where we have only considered the terms in $ \mathcal{O}(\e) $ attached to $\alpha_{\e}.$  \color{black} 

The condition $\int_{\mathbb{T} \times \mathbb{R} } \eta(x,v) \dd x \dd v = 1$ forces the asymptotic relationship
\begin{align}
2 \sqrt{ M }_\delta \alpha_{\e}  \left( 1 + \frac{ m_{1} }{ 4} \e\ + \mathcal{O}\left( \e^{2} \right) \right) + \mathcal{O}\left( \e^{2} \right) = 1 \label{condition-alpha} 
\end{align}
Consider the expansions $\alpha_{\e} = \alpha_{0} + \alpha_{1} \e + \mathcal{O}(\e^{2})$ and $\sqrt{M_\e} = c_{0} + c_{1} \e + \mathcal{O}\left( \e^{2} \right).$  Solving jointly from the equations \eqref{condition-M-0} and \eqref{condition-alpha} yields the relationships $c_{0} = \frac{1}{2}, c_{1} = - \frac{ m_{1} }{8}, \alpha_{0} = \frac{1}{4} $ and $\alpha_{1} = \frac{m_{1}}{8},$ proving the required relationships.

The use of the implicit function Theorem and the fact that $m_1 \neq 0$ yields the last part of the statement. 
%
%
\end{proof}

%
%
%
%
%

\subsection{Penrose criterion} \label{stability-penrose-section}

In order to finish the proof of Theorem \ref{contruction-stat-state}, we shall prove that the  stationary state  $\eta_{\delta}(x, v)$ given by 
Lemma \ref{lemmmzzer} verifies  (H3) for $\delta$ sufficiently small.
\color{black}

\begin{Lemma}[Kernel Stability Lemma] \label{kernel-stability}
Consider for  $\e \in (0,\e_*)$   the family of stationary states $\eta_{\e}$ constructed  in  Lemma \ref{lemmmzzer}.
Then there exists $\delta_{**} \in (0, \delta_*]$ such that for every 
 $\delta \in (0, \delta_{**})$, there exists $c_0>0$ and  $\mu>0$ small enough such that
   \begin{multline} 
\label{condpen2}
\forall\e' \in (\e- \mu,\e+\mu) \qquad 
\inf\limits_{\Im \xi \leq 0} \lvert 1 - \hat{K}_{\Cc}[\eta_{\delta}](\xi) \rvert \geq c_0 > 0, \\
\mbox{with}\quad K_{\mathcal{C}}[\eta_{\delta}](t) = \displaystyle\mathds{1}_{\{t\geq 0\}} \int_{ \mathbb{T} \times \mathbb{R} } \cos(X) \bigl\{ \eta_\delta , \cos(X \circ \varphi_{M_\delta}^{-t}) \bigr\}_{x, v} dx \, dv.  
\end{multline} 
where $\varphi_{M_\delta}^t$ denotes the flow of the Hamiltonian $h_{M_\e}(x,v) = \frac{v^2}{2} - M_\e \cos x$. 
\end{Lemma}

By using again as stated in the last part of  Lemma \ref{lemmmzzer}. that $\delta \mapsto M_\delta$ is invertible, we then get
 that  for the stationary state \eqref{defetaMpreuve}, (H3) in Assumption \ref{hypeq} is verified for $M_0$ sufficiently close to $1/4$
 so that Theorem \ref{contruction-stat-state} is proven. 
 
 Note that by finite covering property, we can actually deduce from Lemma \ref{kernel-stability} that for every $\delta_- \in (0 , \delta_{**}],$ 
 the Penrose condition  holds uniformly  for $\delta \in [\delta_-, \delta_{**}]$.

\begin{proof}
In all the proof, we will consider $\e \in (0,\e_*)$ and 
the corresponding stationary state \(\eta_{\e}(x, v) = \Psi_{\delta}\left[ k_{M_\delta}\big( h_\e(x, v) \big) \right] = G_{\e} \left( h_{M_\e}(x, v) \right) \)  obtained in Lemma \ref{lemmmzzer}.

As in the previous subsection, it is convenient to use the scaling property of $\eta_\delta$ to express  $ \hat{K}_{\Cc}[\eta_{\delta}]$
in terms of $ \hat{K}_{\Cc}[\Psi_\delta(k^2)]$  where $k^2= (h+1)/2$,  $h(x,v)= v^2/2-\cos x$,  
\begin{equation}
\label{noyaumodif}
K_{\Cc}[\Psi_\delta(k^2)] (t) = { 1 \over 2  M_\delta^{1 \over 2}} \displaystyle\mathds{1}_{\{t\geq 0\}} \int_{ \mathbb{T} \times \mathbb{R} } \cos(X) \bigl\{ \Psi_\delta(k^2), \cos(X \circ \varphi^{-t}) \bigr\}_{x, v} dx \, dv, 
\end{equation}
and $\varphi^t$ is the Hamiltonian flow of $h=v^2/2- \cos x$.
  We first observe that since $\varphi^t$ is symplectic, we  have
  \begin{multline*}
K_{\mathcal{C}}[\eta_{\delta}](t) = \displaystyle\mathds{1}_{\{t\geq 0\}} \int_{ \mathbb{T} \times \mathbb{R} } \cos(X\circ \varphi_\e^{t}) \bigl\{ \eta_\delta , \cos X ) \bigr\}_{x, v} dx \, dv.
 \\= \displaystyle\mathds{1}_{\{t\geq 0\}} \int_{ \mathbb{T} \times \mathbb{R} } \cos(X\circ \varphi_\e^{t})  {  v \over 2 M_\delta}  \sin x \, \Psi_\delta'(k_{M_\delta}^2) dx \, dv
\end{multline*}
where $k_{M_\delta}^2 = {h_{M_\delta} + M_{\delta} \over 2 M_\delta}.$
By setting $ \tilde{v}=v/ M_\delta^{1 \over 2}$ and then omitting the tilde, we then obtain
$$ K_{\mathcal{C}}[\eta_{\delta}](t)   =  { 1 \over 2} \displaystyle\mathds{1}_{\{t\geq 0\}}  \int_{ \mathbb{T} \times \mathbb{R} } \cos(X\circ \varphi_\e^{t} (x, M_\delta^{1\over 2} v) )  v   \sin x \, \Psi_\delta'(k^2) dx \, dv $$
where $k^2= {h+1\over 2}$. Finally, an elementary ODE argument based on the  scaling property of the Hamiltonian ODE solved by $\varphi^t_{M_\delta}$ gives that
$$ X\circ \varphi^t_{M_\delta}( x, {M_\delta}^{1 \over 2 } v) = X \circ \varphi^{M_\delta^{1 \over 2} t }(x,v).$$
We thus obtain that
$$
K_{\mathcal{C}}[\eta_{\delta}](t)   = { 1 \over 2}  \displaystyle\mathds{1}_{\{t\geq 0\}}  \int_{ \mathbb{T} \times \mathbb{R} } \cos(X\circ \varphi^{M_\delta^{1 \over 2} t} (x, v) )  v   \sin x \, \Psi_\delta'(k^2) dx \, dv  \\=   M_\delta^{1\over 2} K_{\Cc}[\Psi_\delta(k^2)] (M_\delta^{1 \over 2}  t)
$$
where the last equality comes again from a symplectic change of variable.

Again by a change of variable in the definition of the  Fourier transform, we thus find that
$$ \widehat{K_{\mathcal{C}}}  [\eta_{\delta}](\xi)=  \widehat{K_{\mathcal{C}}} [\Psi_\delta (k^2)]\left( {\xi \over M_{\delta}^{1 \over 2} }\right) .$$
Consequently, it is equivalent to prove \eqref{condpen} or 
\begin{equation}
\label{condpenbis}
\inf\limits_{\Im \xi  < 0} \lvert 1 -  \widehat{K_{\mathcal{C}}} [\Psi_\delta(k^2)] (\xi) \rvert \geq \kappa > 0.
\end{equation}

We shall thus study \eqref{condpenbis}. 

Note that since 
$$ \bigl\{ \Psi_\delta( k^2), \cos(X \circ \varphi^{-t}) \bigr\}_{x, v}  =  { 1 \over 2} \Psi_\delta'(k^2) \{ h, \cos(X \circ \varphi^{-t}) \}_{x,v} ={ 1 \over 2}  {d\over dt} \left[ \Psi_\delta'(k^2)  \cos X\circ  \varphi^{-t}) \right],   $$ we can also write that
\begin{equation}
\label{KCQC}
 K_{\Cc}[\Psi_\delta(k^2)](t)
 = { d\over dt}   Q_{\Cc}[\Psi_\delta(k^2)](t)
 \end{equation}
 where
 \begin{equation}
 \label{QCdef}
 Q_{\Cc}[\Psi_\delta(k^2)](t) = { 1 \over 4 M_\delta^{1 \over 2} }  \int_{ \mathbb{T} \times \mathbb{R} }  \Psi_\delta'\big(\frac{ h+1}{2}\big) \cos x \cos X \circ \varphi^{-t} dx dv -  Q_0.
 \end{equation}
and where by using  the action angle coordinates of the Hamiltonian $h$ and the    Fourier coefficients of 
 $ \cos x(\theta, a)$  as in \eqref{scaledcos}, we define
 \begin{equation}
 \label{defQ0stab} Q_0 =  { 1 \over 4 M_\delta^{1 \over 2} }  \int_{J_\circ} \Psi_\delta'( k(a)^2)  |C_{0} (a)|^2 \, \dd a
 \end{equation}
 so that 
 \begin{multline}
 \label{QC2}
 Q_{\Cc}[\Psi_\delta(k^2)](t) ={  1 \over 4   M_\delta^{1 \over 2} } \sum_{{\ell\neq 0}} \int_{ J_\circ} \Psi_\delta'( k(a)^2)  |C_{\ell} (a)|^2 e^{-i\ell t \omega(a)} \, da \\= 
 {  1 \over  4 M_\delta^{1 \over 2} } \sum_{{\ell\neq 0}} \int_{ J_\circ} \Psi_\delta'( k(a)^2)  |C_{\ell} (a)|^2 e^{i\ell t \omega(a)} \, da.
 \end{multline}
 The last equality is obtained by changing $\ell$ in $-\ell$ and by noticing that $|C_{\ell}|^2=|C_{-\ell}|^2$ since they are the Fourier coefficients of a real function.
 We recall that $J_\circ$ is the interval $[0, 8/\pi]$.
 
 From, \eqref{KCQC}, we then get that
 \begin{equation}
 \label{KCstab2}
 K_{\Cc}[\Psi_\delta(k^2)](t)= {  1 \over 4  M_\delta^{1 \over 2} } \sum_{{\ell\neq 0}} \int_{ J_\circ} i\ell \omega(a) \Psi_\delta'( k(a)^2)  |C_{\ell} (a)|^2 e^{i\ell t \omega(a)} \, da.
 \end{equation}
 
 By using this expression and  \eqref{deriv1},  we  can obtain  from Lemma \ref{time-decay-ipp} a decay estimate for $ K_{\Cc}[\Psi_\delta(k^2)](t)$.
 Indeed for  every $0 <\delta_- < \delta_*$ and $\delta \in (\delta_-, \delta_*]$, we have by construction that there exists  
 an interval $I_{\delta_-}$  included in the interior of $J_\circ$ such that $\mbox{ Supp}( \Psi_\delta' \circ( k^2) )\subset I_{\delta_-}$
 for every $\delta \in [\delta_-, \delta_*]$. Moreover,  the distance of  $I_{\delta_-}$ to the boundary of  $J_\circ$ 
 is bounded from below by $r_{\delta_-}$ which depends only on $\delta_-$.    On  $I_{\delta_-}$, $\omega$ is smooth and 
 since $C_l(a)$ are the Fourier coefficients of a smooth function on  $\mathbb{T} \times I_{\delta_-}$, they are smooth and fastly decaying in $l$.
 We thus obtain by using also \eqref{deriv1}  that
\begin{equation}
\label{borneQC}
  |K_{\Cc}[\Psi_\delta(k^2)](t)| \leq {C_{\delta_-, L}  \over \langle t \rangle^{L} }, 
  \end{equation}
  for any $L\geq 0$ and 
  for  some $C_{\delta_-, L}>0$  which is  uniform for  $\delta \in  [\delta_-, \delta_*]$ and $t$.
  
  Note that this proves that  $K_{\Cc}[\Psi_\delta(k^2)] $ is an $L^1$ function and thus that its Fourier transform
  is well defined by a convergent integral and  is a continuous function.
  
  Thanks to \eqref{deriv1} and the same arguments, we also obtain that for every $\delta', \, \delta  \in [\delta_-,  \delta_{*}],$ we have
  $$  |K_{\Cc}[\Psi_\delta(k^2)](t) -  K_{\Cc}[\Psi_{\delta'}(k^2)](t)|\leq  C_{\delta_-} | \delta'- \delta|.$$
  This yields  that for  $\mu>0$, sufficiently small, there exists  $C>0$,  such that  for every $\delta' \in (\delta -\mu, \delta + \mu)$ 
  $$\sup\limits_{\Im \xi \leq  0} \Big| \widehat{K_{\mathcal{C}}} [\Psi_{\delta'}(k^2)] (\xi) -  \widehat{K_{\mathcal{C}}} [\Psi_\delta(k^2)] (\xi) \Big|
   \leq \mu C_{\delta_-}.$$
   We thus obtain that the Penrose condition is an open condition:   if  \eqref{condpenbis} holds  true for some $\delta$
    then it is also true for $\delta' \in (\delta -\mu, \delta + \mu)$ for $\mu$ sufficiently small (depending on $\delta$).
    We thus only need to prove \eqref{condpenbis} for some fixed $\delta$ sufficiently small.
    For notational convenience, we shall thus set
    $$ K_\mathcal{C} (t)= K_{\Cc}[\Psi_\delta(k^2)](t), \quad  \widehat{ K_\mathcal{C}} (\xi)
    = \widehat{K_{\mathcal{C}}} [\Psi_{\delta}(k^2)] (\xi).
    $$
        
    From \eqref{borneQC},  we have that $K_\mathcal{C}$ is an $L^1 $function and thus, we obtain that $\widehat{K}_\mathcal{C}$ is a continuous function on $\{\Im \xi \leq 0\}.$ Moreover, from a Riemann Lebesgue type argument, we get that
    $$\lim_{| \xi | \rightarrow +\infty, \, \Im \xi \leq 0} \widehat{K}_\mathcal{C} = 0.$$
    This yields  since $\widehat{K}_\mathcal{C}$ is continuous  that
 $  \inf\limits_{\Im \xi \leq 0} \lvert 1 - \hat{K}_{\Cc}(\xi) |$ is reached  and thus it suffices to prove that 
 $ 1 - \hat{K}_{\Cc}(\xi)$ does not vanish for every $\xi, \, \Im \xi \leq 0.$
 
 We shall handle differently  $\Im \xi <0$ and $\Im \xi=0$.
 
 \textbf{Case $\Im \xi <0$.}
 
 Thanks to \eqref{KCstab2},  we get that  for $\Im \xi <0$, 
 \begin{multline*}
 K_\mathcal{C} (t) = {  1 \over 4  M_\delta^{1 \over 2} } \sum_{{\ell\neq 0}} \int_{ J_\circ} i\ell \Psi_\delta'( k(a)^2)  |C_{\ell} (a)|^2 
  \int_0^{+\infty} e^{-it ( \xi- \ell \omega(a)) } \, \dd s \, da \\
  = {  1 \over 4   M_\delta^{1 \over 2} } \sum_{{\ell\neq 0}} \int_{J_\circ} {\ell \over (\xi- \ell \omega(a))}  \Psi_\delta'( k(a)^2)  |C_{\ell} (a)|^2.
 \end{multline*}
  so that we get for $\xi = \gamma + i \tau, $ $\tau<0$.
  \begin{align}
 \label{Realdef}
 & \mathrm{Re} \, \hat K_\Cc(\xi) = \frac{1}{ 4  M_\delta^{1 \over 2}} \sum_{\ell \neq 0} \int_{J_{\circ}}  \Psi_{\delta}'(k) |C_\ell(a)|^2  \left[\frac{(\gamma- \ell \omega(a))\ell \omega(a)}{|\gamma - \ell \omega(a)|^2+\tau^2}\right] \dd a, \\
\label{Imdef} & \mathrm{Im} \, \hat K_\Cc (\xi) = - \frac{1}{ 4 M_\delta^{1 \over 2}} \sum_{\ell \neq 0} \int_{J_{\circ}} \Psi_{\delta}'(k) |C_\ell(a)|^2 \tau \ell \omega(a) \frac{1}{[\gamma - \ell \omega(a)]^2 + \tau^2} da.
\end{align}
Assuming that  $1 - \hat{K}_{\Cc}(\xi)=0$, we get the vanishing of  $\Im \hat{K}_{\Cc}(\xi)$ and of $1- \Re \hat{K}_{\Cc}(\xi)$.
From the vanishing of the imaginary part, we get that
$$ \sum_{\ell \neq 0} \int_{J_{\circ}} \Psi_{\delta}'(k^2) |C_\ell(a)|^2 \ell \omega(a) \frac{1}{[\gamma - \ell \omega(a)]^2 + \tau^2} = 0 $$
and plugging this identity into the real part we get that
$$ 1-  \Re \hat K_\Cc(\xi) =  1 - \frac{1}{ 4  M_\delta^{1 \over 2}} \sum_{\ell \neq 0} \int_{J_{\circ}}  \Psi_{\delta}'(k^2) |C_\ell(a)|^2  \ell^2 \omega(a)^2\left[\frac{1}{|\gamma - \ell \omega(a)|^2+\tau^2}\right] \dd a
\geq 1$$
since $-\Psi_\delta'$ is a nonnegative function. This yields a contradiction.

We have thus obtained that $1 - \hat{K}_{\Cc}(\xi)$ does not vanish for $\Im \xi <0$.

 \textbf{Case $\Im \xi=0$.}
 
 If $\xi \in \mathbb{R}$, we  shall handle differently $\xi$ small and $\xi$ large.

 From \eqref{KCQC}, we  first observe that we can  integrate by parts to get  for $\Im \xi>0$ to get that  that
\begin{align} 
\label{ipetit1}
 \hat K_\Cc (\xi)= -  Q_{\mathcal{C}} (0)  +  \frac{i\xi}{4   M_\delta^{1 \over 2 }} \int_{0}^{+\infty} e^{- i t \xi} Q_\mathcal{C} (t) \dd t 
 =  -  Q_{\mathcal{C}} (0) + i \xi  \widehat{Q}_{\mathcal{C}}(\xi).
\end{align} 
By using the expression of $Q_\mathcal{C}$ given by \eqref{QCdef}, we can get a decay estimate with a precise dependence in $\delta$.
Indeed, we can use \eqref{dampend} with $f =\chi(v)  \Psi_{\delta}'\big[ { |v|^2 \over 4}  + { 1 \over 2}(1-\cos x)\big] \cos x$ and
 $\psi(x,v)= \chi(v) \cos x $ where $\chi$ is a smooth compactly supported function on $\mathbb{R}$ which is $1$ on $[-2, 2]$.
 With this choice, we have $p=1$, $q=1$. Moreover, by using the precise  quantitative version of \eqref{dampend} stated 
 in  \cite{Faou2021-qm} Theorem 2.2,  we get that
 $$ |Q_\mathcal{C} (t)| \leq { C \over \langle t\rangle^3 }  \Norm{ \Psi_\delta'[k^2(h)](x,v)}{W^{7, \infty}_{x,v}} \leq { C \over \delta^{9} \langle t\rangle^3} $$
 where the last estimate comes from \eqref{deriv1}. Here $C>0$ is independent of $\delta\in (0, \delta_*].$
 We thus deduce that
 $$ \sup_{ \Im \xi \leq 0 } |\widehat{Q}_{\mathcal{C}}(\xi) | \leq { C \over \delta^{9} }.$$
  We thus obtain from \eqref{ipetit1} that  for every $\xi$, $\Im \xi <0$
  $$  |\hat K_\Cc (\xi)| \leq | Q_{\mathcal{C}} (0) | + | \xi| { C \over \delta^{9}  }$$
  where $C$ in  the right hand side is independent of $\xi$ and $\delta$.  By setting $\xi= \gamma + i \tau$, we can take the limit when $\tau$ tends to
  zero and use the continuity of $\hat K_\Cc$ in $\{\Im \xi \leq 0 \}$ to deduce that
  \begin{equation}
  \label{KCgammapetit}
    |\hat K_\Cc (\gamma )| \leq | Q_{\mathcal{C}} (0) | + |\gamma| { C \over \delta^{9}  } 
    \end{equation}
  where $C>0$ is independent of $\gamma$ and $\delta$.
 It remains to estimate  $ | Q_{\mathcal{C}} (0) |$.
  We  switch to the variable $k$  in the expression of $Q_\mathcal{C} (0)$ given by \eqref{QC2}.  Thanks to \eqref{chgh}  and \eqref{bigom}, we get
$$
        Q_{\Cc}(0) = \frac{2  }{\pi  M_\delta^{1 \over 2}} \sum_{\ell \neq 0} \int_{0}^{1}   \Psi_\delta' \left( k^2 \right) \left| C_\ell(k) \right|^2 \bK(k)  k d k
  $$      
where we use the notation  $C_\ell(k)$
as a shorthand for  $C_\ell[a(h(k))]$ with $h(k) = 2k - 1$. 
    Using \eqref{dqsdoi} we have (as $\beta = \e^2$) 
$$
\Psi_\delta'(z) = \frac{\alpha_\e}{\e^2}\phi'\big(\frac{z}{\e}\big)+ \frac{\e^2}{\nu_{2}(\e) - \nu_{1}(\e)}  \chi'\big(\frac{z - \nu_{1}(\e)}{\nu_{2}(\e) - \nu_{1}(\e)}\big). 
$$ 
The first part has a support included in $z = k^2 \in (0,\e)$, the second part in $z = k^2 \in (1- \delta,1)$ and the function $\phi'$ and $\chi'$ are bounded on these intervals. 
As $\alpha_\e$ and $M_\e$ are uniformly bounded with respect to $\e$, see \eqref{alphamzero}, and as $\frac{1}{\nu_{2}(\e) - \nu_{1}(\e)} \leq \frac{1}{\e}$, we can thus bound 
\begin{eqnarray*}
|Q_{\Cc}(0)|&\leq&C\sum_{\ell \neq 0} \frac{1}{\e^2}\int_{0}^{\sqrt{\e}}    \left| C_\ell(k) \right|^2 \boldsymbol{K}(k)   k \dd k
+ C \sum_{\ell \neq 0} \frac{\e^2}{\e}\int_{\sqrt{1- \e}}^{1}    \left| C_\ell(k) \right|^2  \boldsymbol{K}(k)   k \dd k.
\end{eqnarray*}   
 We have  the expressions \ref{coeffpaszero}, \eqref{coeffimpair}  and    the asymptotics (see for example  \cite[Proposition~7.1]{Faou2021-qm} and the references therein):
\begin{equation*}
\left|
\begin{array}{ll}
\displaystyle\bK(k)  \sim \frac{\pi}{2} \left( 1 + \frac{1}{4} \, k^{2} \right), & k \to 0\\
\displaystyle\bK(k) \sim - \frac12 \log (1 - k) & k \to 1
\end{array}
\right.
\qquad \mbox{and} \qquad 
\left|
\begin{array}{ll}
\displaystyle q(k) \sim \frac{k^2}{16}& k \to 0\\
\displaystyle q(k) \sim 1 + \frac{\pi^2}{\log (1 - k)} & k\to 1. 
\end{array}
\right.
\end{equation*}
As  $q < 1$, we have $\frac{1}{1 - q^{2|n|}} \leq \frac{1}{1 - q^2}$ and thus, we deduce that  
\begin{align}
\label{marre}
\sum_{\ell \neq 0 } |C_\ell^\circ(k)|^2 &\leq \frac{C}{\bK(k)^4 [1 - q(k)^2]} 
\sum_{n > 1} n^2 q(k)^{2n}  \leq C \frac{q(k)^2 }{\bK(k)^4[1 - q(k)^2]^4}
\leq C
\left\{
\begin{array}{ll}
 k^4 & k \to 0, \\
 1 & k \to 1, 
\end{array}
\right.
\end{align}
for some numerical constants $C$. 
Hence we can bound 
\begin{equation}
\label{bounds-Q-0}
|Q_{\Cc}(0)| \leq C \frac{1}{\e^2}\int_{0}^{\sqrt{\e}}   k^5 \dd k
+ C \e\int_{\sqrt{1- \e}}^{1} |\log(1 - k)| k \dd k \leq  \e^{1 \over 2},  
\end{equation}
for some constants $C$ independent of $\e$. 

We thus get from \eqref{bounds-Q-0} and \eqref{KCgammapetit} that
$$ |\widehat{K}_{\mathcal{C}}(\gamma) |\leq \delta^{ 1\over 2}, \, \quad |\gamma | \leq \delta^{10}.$$
Therefore, for $\delta$ sufficiently small, $ 1- \widehat{K}_{\mathcal{C}}(\gamma)$ cannot vanish for $|\gamma | \leq \delta^{10}.$

It remains to study $|\gamma | \geq \delta^{10}.$ 
We shall prove that   $ \Im K_\mathcal{C}$ does not vanish in this range of parameter. Let us start from \eqref{Imdef} valid for $\tau<0$, 
by changing variable to $k$ (thus by using again  \eqref{chgh} for $M=1$) we obtain:
   $$            \mathrm{Im} \, \hat K_\Cc  (\xi) = - \frac{1}{M_\delta^{ 1 \over 2}} \sum_{\ell \neq 0} \int_{0}^1 \Psi_{\delta}'(k^2) |C_\ell(k)|^2 \tau \ell\frac{1}{[\gamma - \ell \omega(k)]^2 + \tau^2} \, k dk 
    $$
    and thus 
    $$  \mathrm{Im} \, \hat K_\Cc (0) = \lim_{\tau \rightarrow 0^-}  \frac{1}{M_\delta^{ 1 \over 2}} \sum_{\ell \neq 0} \int_{0}^1 [- \Psi_{\delta}'(k^2)] |C_\ell(k)|^2 \tau \ell\frac{1}{[\gamma - \ell \omega(k)]^2 + \tau^2} \, k dk. 
    $$

By symmetry in $\ell$ in the sum, we can only consider the case $\gamma>0$. We first prove that the sum for negative $\ell$ does not contribute
to the limit.
 Indeed,  we have for $\ell<0$, since $\omega$ is a nonnegative function that  $\gamma - \ell\omega(a) >\gamma$ and thus  $\frac{1}{ \left[ \gamma - \ell \omega(k) \right]^{2} + \tau^{2}} \leq \frac{1}{ \gamma^{2}}$. Hence we can bound as 
\begin{multline*}
\left|\sum_{\ell < 0} \int_{0}^1 -\Psi_{\delta}'(k^2) |C_\ell^\circ(k)|^2 \tau \ell \frac{1}{[\gamma - \ell \omega(k)]^2 + \tau^2} \,k dk\right| 
\\ \leq
\frac{|\tau|}{\gamma^2} \int_0^1\Psi_{\delta}'(k^2) \left(\sum_{\ell< 0}|\ell |C_\ell^\circ(k)|^2 \right)  kdk \leq C_\delta |\tau|
\end{multline*}
for some  $C_\delta>0$ depending on   $\e$ by using again the smoothness and decay of the $C_\ell$ on the support of $\Psi_\delta'$, the estimate \eqref{deriv1} and the fact that we now consider the case $\gamma \geq \delta^{10}.$
Consequently the left-hand side goes to zero when $\tau$ tends to zero.
We thus obtain that
\begin{equation}
\label{ladernierestab}  \mathrm{Im} \, \hat K_\Cc (0) = \lim_{\tau \rightarrow 0^-}  \frac{1}{M_\delta^{ 1 \over 2}} \sum_{\ell > 0} \int_{0}^1 [- \Psi_{\delta}'(k^2)] |C_\ell^\circ(k)|^2 \tau \ell\frac{1}{[\gamma - \ell \omega(k)]^2 + \tau^2} \, k dk 
   \end{equation}
Note that in the right hand side, we now have a sum of nonpositive terms. So that, we only need to find in the sum one term which has
a strictly negative limit when $\tau$ tends to zero.
By construction, thanks to \eqref{propchi} and \eqref{aedmoqdij}, we have that $\Psi_\delta' <0$ for
 $ k^2 \in [\nu_1(\delta),  \nu_2(\delta)]$.
   Thanks to \eqref{omega'},
   $$ \omega: \,  k \in  \left[ (1 - { \delta \over 2})^{ 1 \over 2},  (1 - 2 e^{- { 1 \over \delta^{20} } } )^{ 1 \over 2}\right]  \rightarrow   [c_\delta, d_\delta]$$
   is a strictly decreasing diffeomorphism.  By using     \cite[Proposition~7.1]{Faou2021-qm} (see also formula 900.05 of \cite{byrd1971handbook}), 
   we have 
   $$
        \omega(k) =  -\frac{  \pi }{\log (1 - k)} \big[ 1 + \mathcal{O}(1 - k)\big], \quad \mbox{when} \quad k \to  1,
$$
    so that we get\footnote{Here and in the sequel, we use $A \lesssim B$ for $A \le C B$ where $C$ is a constant independent of the varying parameters, and similar convention for $A \gtrsim B$ and $A \sim B$}:
     \begin{equation}
      \label{eqlog}
        c_\delta \sim  \pi  \delta^{20}, \quad  d_\delta \sim - {\pi  \over   \log \delta}.
     \end{equation}
     Our goal is to be able to find for every $\gamma \geq \delta^{10}$   an even integer $\ell_*= 2 p^*\geq 2$ such that $\gamma/ \ell_*
      \in  [c_\delta, d_\delta] $.  Since  this is equivalent to $ \ell_* \in  [ \gamma/d_\delta, \gamma/c_\delta],$ this is possible as soon as the length of this interval which is $\gamma( 1/c_\delta - 1/d_\delta)$ is bigger than 2.
         Since we are in the regime $\gamma \geq \delta^{10} $ this length is always bigger than $L_\delta=  \delta^{10}( 1/c_\delta - 1/d_\delta)$
         which depends only on $\delta$.
         Thanks to  \eqref{eqlog} we have that $\lim_{\delta \rightarrow 0} L_\delta=+\infty$  and thus we take $\delta$ sufficiently small
         so that $L_\delta \geq 2$.
         
         We thus now fix $\ell_*=2p_*\geq 2$ (which depends on $\gamma$ and $\delta$) so that $\gamma/\ell_* \in [c_\delta, d_\delta]$. This allows to 
         define $v_* \in [ (1 - { \delta \over 2})^{ 1 \over 2},  (1 - 2 e^{- { 1 \over \delta^{20} } } )^{ 1 \over 2}] $ such that $\omega(v^*) = \gamma/\ell_*$.
         With this choice,  we have $\Psi_\delta' (v_*^2) <0$. We have from \eqref{ladernierestab} that
         \begin{equation} 
         \label{Lpenrose} \mathrm{Im} \, \hat K_\Cc (\gamma) \leq  \lim_{\tau \rightarrow 0^-}  \mathcal{L}_\tau
         \end{equation}
         where we set
  $$\mathcal{L}_\tau=
          \frac{1}{M_\delta^{ 1 \over 2}}  \int_{0}^1 [- \Psi_{\delta}'(k^2)] |C_{\ell_*}(k)|^2 {\tau\over \ell_*}  \frac{1}{[{\gamma\over \ell_*} - \omega(k)]^2 + ({\tau\over \ell_*})^2} \, k dk.$$ 
     Let us also set for notational convenience
     $$ F(k) =- { 1 \over M_\delta^{ 1 \over 2}} \Psi_{\delta}'(k^2)) |C_{\ell_*}(k)|^2 k  { 1 \over \omega'(k) }.$$
     Note that  since  $C_{\ell_*}$ does not vanish  thanks to \eqref{coeffpaszero} and $\omega'<0$, we have that $F(v_*)<0$.
      To compute the limit, 
     we then split 
     $$ \mathcal{L}_\tau =  F(v_*) \int_0^1 {\tau\over \ell_*}  \frac{1}{[{\gamma\over \ell_*} - \omega(k)]^2 + ({\tau\over \ell_*})^2} \omega'(k) \,  dk
     + \mathcal{R}_\tau$$
     where 
     $$  \mathcal{R}_\tau = \int_0^1 {\tau\over \ell_*}  \big[F(k)- F(v_*)\big] \frac{1}{[{\gamma\over \ell_*} - \omega(k)]^2 + ({\tau\over \ell_*})^2} \omega'(k) \,  dk.$$
     By taking $\omega$ as a new variable, we get
     \begin{multline*}
      \int_0^1 {\tau\over \ell_*}  \frac{1}{[{\gamma\over \ell_*} - \omega(k)]^2 + ({\tau\over \ell_*})^2} \omega'(k) \,  dk
     = \int_{0}^{\omega(0)}   \left|{\tau\over \ell_*} \right| \frac{1}{[\omega(v_*) - \omega(k) ]^2 + \left|{\tau\over \ell_*}\right|^2} \, \dd \omega
     \\ = \arctan\left( { \omega(0) - \omega_* \over \left| { \tau \over \ell_*} \right| }\right) - \arctan\left( {- \omega_*  \over \left| { \tau \over \ell_*} \right| } \right)\xrightarrow[ \tau \to  0^-]{}  \pi.
     \end{multline*}
     To estimate $\mathcal{R}_\tau$, we  use that $F$ is smooth and that $\omega'$  does not vanish  on the support of $F$, this yields
     $$  |F(k) - F(v_*)| \leq C_\delta    |k - v_*| 
    \leq C_\delta \left| \omega(k) - { \gamma \over \ell_*}\right|, \quad \forall k\in \mbox{Supp } \Psi_\delta'.$$
    Therefore, by taking again $\omega$ as new variable, we obtain
    $$
     |\mathcal{R}_\tau |  \leq  C_\delta
     \int_{0}^{\omega(0)}  \left|{\tau\over \ell_*}  \right|  \left| \omega - { \gamma \over \ell_*}\right|  \frac{1}{(\omega - { \gamma \over \ell_* })^2 + \left|{\tau\over \ell_*}\right|^2} \, \dd \omega 
     \leq C_\delta    \left|{\tau\over \ell_*}  \right|^{ 1\over 2}    \int_{0}^{\omega(0)}  { 1 \over \left| \omega - { \gamma \over \ell_*}\right|^{1 \over 2}}
      \, \dd \omega  \xrightarrow[ \tau \to  0^-]{}  0.
  $$
  We have thus established that
  $$\lim_{\tau \rightarrow 0^-} \mathcal{L}_\tau = \pi  F(v_*) <0$$ 
  so that $\ \mathrm{Im} \, \hat K_\Cc (\gamma)$ cannot vanish thanks to \eqref{Lpenrose}.

\end{proof}

\section{Proof of Theorem \ref{thnonlin}} \label{equations-section}

We will now start the proof of Theorem \ref{thnonlin}. 
Note that for the Vlasov HMF model, the global existence of a smooth solution is trivial. The main difficulty is to estimate the solution.
We first fix 
\begin{equation}
\label{hypkappa}
0<\kappa_0  < { 1 \over 4} \overline{\kappa}_0
\end{equation}
 and we shall always work on an interval of time  where $f - \eta_{M_0}$ is compactly supported  in $\mathcal{K}_{M_0}$ defined in \eqref{defKM0}.
 We consider  $\mu>0$ sufficiently small  such that  $[M_0- \mu, M_0+ \mu ] \subset I_0$, $I_0$ given by assumption \ref{hypeq} and
 such that for  $M\in [M_0- \mu, M_0+ \mu ]$, 
   \eqref{K0inclus} holds  and hence 
 \begin{equation}
 \label{supportG}
 \mbox{ Supp }G_M \subset [-M, M_0 - \kappa_0], \quad   \mbox{ Supp }G_M' \subset [-M_0+\kappa_0, M_0 - \kappa_0]
 \end{equation}
 thanks to the choice \eqref{hypkappa}. 

The first step is to use modulation theory to select in a  well chosen dynamical way    the magnetization  parameter of the steady state.

\subsection{Choice of the modulation parameter}
\begin{lemma}
\label{lemmod}
For every  $\overline{M}_{0} \in (M_{0}- \mu, M_{0}+\mu)$,  there exists $\rho_{0}>0$, $C_{0}>0$ such that for every   $f(x,v) \in \mathcal{C}^0(\mathbb{T} \times \mathbb{R})$ with $\Norm{f - \eta_{\overline{M}_{0} }}{\mathcal{C}^0} \leq \rho_{0} $ 
 and such that $\mbox{Supp }f - \eta_{\overline{M}_{0}} \subset \mathcal{K}_{M_0}$, 
there exists a $\mathcal{C}^1$ maps  $f\mapsto M_{f}>0$ with 
\begin{equation}
\label{M-M0}
 |M_{f}-\overline{M}_{0}| \leq C_{0} \Norm{f - \eta_{\overline{M}_{0}}}{\mathcal{C}^0} 
 \end{equation} such that
\begin{equation}
\label{ortho}
\int_{\mathbb{T}\times[a_{-}, a_{+}]} \left[ \left( f - \eta_{M_{f}} \right)\circ \Phi_{M_{f}}(\theta, a) \right]  C_{0,M_f}(a, M_{f}) \, d\theta da=0 
\end{equation}
where $[a_{-}, a_{+}]$. \end{lemma}

Note that the condition \eqref{ortho} can be seen as an orthogonality condition in $L^2_{a}$ between 
the angle average $ \int_\mathbb{T} \left(  f - \eta_{M}\right) \circ \Phi_{M}(\theta, a) d\theta$ and $C_{M, 0}(a)$
which is the angle average of the function $\cos x_{M}(\theta, a)$ (see \eqref{scaledcos}).

\begin{proof}
As usual the proof is based on the Implicit Function Theorem. We set $f =\eta_{\overline{M}_0} +g$ and consider 
$$ F(g, M)= \int_{\mathbb{T}\times[a_{-}, a_{+}]} \left[ \left( \eta_{\overline{M}_0}  - \eta_{M} + g  \right) \circ \Phi_{M}(\theta, a) \right] C_{0,M}(a) \, d\theta da.$$
We observe that $F(0, \overline{M}_{0})= 0$ and that $F$  defined on  
$U \times  I $ (where $U$ is  an open  vicinity of  zero in the Banach space of  $\mathcal{C}^0$ functions
compactly supported in $\mathcal{K}_{M_0}$ and $I$ is a small open interval containing $\overline{M}_{0}$)
 is  a $\mathcal{C}^1$ function thanks to Lemma \ref{lemMaction}.  Indeed, switching back to $(x,v)=\Phi_{M}(\theta, a) $ which is symplectic, 
 we can rewrite
 $$  F(g, M)= \int_{{K}_0}   \left[ \eta_{\overline{M}_0}(x,v)  - \eta_{M}(x,v) + g(x,v)  \right]  C_{0,M}\big(a_M(x,v)\big) \, \dd x \dd v$$
and we see on this expression that $F$ is actually smooth.
 
 We then obtain  from the original definition that 
$$ \partial_{M} F( 0, \overline{M}_0)=  -\int_{\mathbb{T}\times[a_{-}, a_{+}]}  \left({\partial \eta_{M} \over \partial {M}}\right)_{\vert M=\overline{M}_0 }
\circ \Phi_{\overline{M}_0 }(\theta, a)\,  C_{0,\overline{M}_0}(a) d\theta da.  $$
Next, by using that $\eta_{M}(x,v)= G_{M}( h_{M}(x,v))$, we obtain that
\begin{equation}
\label{eqMderivative}  \left({\partial \eta_{M} \over \partial{M}}\right)_{\vert M=\overline{M}_0 } =  \left({\partial G_{M} \over \partial{M}}\right)_{\vert M=\overline{M}_0 }
\circ H_{\overline{M}_0} (a) - \left(G_{\overline{M}_0 }'\circ H_{\overline{M}_0}(a) \right) \cos x 
\end{equation}
so that
\begin{multline}
\label{FM1}
  \partial_{M} F(0 , \overline{M}_0 )= \\  -\int_{\mathbb{T}\times[a_{-}, a_{+}]} 
  \left\{ 
\left[  \left({\partial G_{M} \over \partial{M}}\right)_{\vert M=\overline{M}_0 }
\circ H_{\overline{M}_0 }(a)\right]  C_{0,\overline{M}_0 }(a) - G_{\overline{M}_0 }'\circ H_{\overline{M}_0 }(a) \,  C_{ 0,\overline{M}_0 }(a)^2  \right\} d\theta da
\end{multline}
by using the Fourier series expansion \eqref{scaledcos}.

Next,  since  for every $M \in I_0$ given by \eqref{hypeq}, we have 
$$ M= \int_{\T \times \R} \eta_{M}(x,v) \cos x \, dx dv$$
we can differentiate with respect to $M$ to get that
$$ 1 =  \int_{\T \times \R}  \left({\partial \eta_{M} \over \partial{M}}\right)_{\vert M=\overline{M}_{0}} (x,v) \cos x \, dx dv.$$
Next, by using again \eqref{eqMderivative} and by making the symplectic change of variable $(x,v)= \Phi_{M_{0}}(\theta, a)$, 
we get that
\begin{multline*} 1 =  \int_{\mathbb{T}\times[a_{-}, a_{+}]} \left[  \left({\partial G_{M} \over \partial{M}}\right)_{\vert M=\overline{M}_{0}}
\circ H_{\overline{M}_{0}}(a)\right]    C_{0,\overline{M}_{0}}(a) \, d\theta da \\
- \int_{\mathbb{T}\times[a_{-}, a_{+}]}  \left(G_{\overline{M}_{0}}'\circ H_{\overline{M}_{0}}(a)\right) (\cos x_{\overline{M}_{0}}(\theta, a))^2 \, d\theta da.
\end{multline*}

By plugging this identity into \eqref{FM1},  we thus obtain that
\begin{multline*}
 \partial_{M} F( 0, \overline{M}_{0}) = - \Bigl(  1 +  \int_{\mathbb{T}\times[a_{-}, a_{+}]}  \left(G_{\overline{M}_{0}}'\circ H_{\overline{M}_{0}}(a)\right) \left(  (\cos x_{\overline{M}_{0}}(\theta, a))^2 - C_{0,\overline{M}_{0}}(a)^2 \right) \, d\theta da \Bigr) \\
  = - \Bigl( 1 +  \sum_{\ell \neq 0} \int_{\mathbb{T}\times[a_{-}, a_{+}]}  \left(G_{\overline{M}_{0}}'\circ H_{\overline{M}_{0}}(a)\right)  \Bigr)
   |C_{\ell,\overline{M}_{0}}(a)|^2 \, da.
\end{multline*}
By using the computations in section \ref{stability-penrose-section}, in particular,  \eqref{KCQC}, \eqref{QC2}, we obtain  that 
$$  \partial_{M} F( 0, \overline{M}_{0}) = - \Big( 1 + Q_{\Cc,\overline{M}_0}(0)\Big) = - (1 - \widehat{K}_{\mathcal{C}} [\eta_{\overline{M}_0 }](0))\neq 0$$
thanks to the Penrose criterion, where $ Q_{\Cc,\overline{M}_0}(0):= Q_{\mathcal{C}}[\Psi_{\delta(\overline{M}_0)}(k^2)](0)$ is defined in \eqref{defetaMpreuve}.

The statement then follows from  the Implicit Function Theorem.
\end{proof}

\subsection{Initialisation and bootstrap set up}

Assuming that $\eps_{0}$ is smaller than $\rho_{0}$, we can use Lemma \ref{lemmod} at the initial time. We can thus write
\begin{equation} f_{0}(x,v)= \eta_{M_{f_{0}}} + R^{0}(x,v)$$ where
 $R^{0}$ and $M_{f_0}$ are now such that 
 $$ \int_{\mathbb{T}\times[a_{-}, a_{+}]} \left( R^{0} \circ \Phi_{M_{f_{0}}} \right) C_{0,M_{f_{0}}} \, d\theta da = 0 \quad\mbox{and}\quad  \label{m0petit}
    |M_{f_{0}} - M_{0}| \leq C_{0} \eps.
    \end{equation}
  For $\eps_{0}$ sufficiently small this yields thanks to \eqref{hypkappa}  that $R^{0}$ is compactly supported in $\mathcal{K}_0$  (we shall give more justification for this    in the end of  Section \ref{cestfini} this is why we do not reproduce it here)
  and that  for some $\mathcal{R}_{0}>0$, we have
\begin{equation}
\label{r0petit}
\Norm{ R^0 }{H^{\sigma + 1}} \leq \mathcal{R}_0 \eps.
\end{equation}
We can thus define for the solution $f(t)$ of \eqref{eq:VlasovHMFmodel}, 
\begin{multline}
\label{T*def}
T^*= \sup \Big\{  T \geq 0, \mbox{ such that } \, |M_{f(t)}-M_{0}| \leq 2 C_{0} \eps, \,  \\
\|f(t)-\eta_{M_{f(t)}}\|_{C^0} \leq \rho_{0}, \, f- \eta_{M_{0}} \mbox{ supported  in } \mathcal{K}_{M_{0}}   \Bigr\}.
\end{multline}
By continuity  and thanks to Lemma \ref{lemmod}, we have that $T^*>0$ and  for every $t \in [0,  T^*)$, we can set 
\begin{equation}
\label{decdyn} f(t,x,v)= \eta_{M(t)} + R(t,x,v), \quad M(t):=M_{f(t)}, 
\end{equation}
where  $R(t)$ is compactly supported in $\mathcal{K}_{M_{0}}$  and  such that
\begin{equation}
\label{orthor}
 \int_{\mathbb{T}\times[a_{-}, a_{+}]} \left( R(t, \cdot) \circ \Phi_{M(t)} \right) C_{0,M(t)} \, d\theta da = 0. 
\end{equation}
We shall  set 
\begin{equation}
\label{defpetit}
 M(t) = M_{0}+ \eps m(t), \quad R (t) = \eps  r (t), \quad t \in [0, T^*)
\end{equation}
and  $\zeta(t) = \mathcal{C}[ r] (t)$
so that thanks to \eqref{decdyn}, we have
\begin{equation}
\label{zetadef}
\mathcal{C}[f](t) = M(t) + \eps \zeta(t)= M_{0}+ \eps m(t) + \eps \zeta(t).
\end{equation}
Note that initially, we have thanks to \eqref{r0petit}, \eqref{m0petit}, we have that
\begin{equation}
\label{petitinit} |m(0)| \leq 2C_{0}, \quad \|  r(0)\|_{H^{\sigma + 1}} \leq \mathcal{R}_{0}.
\end{equation}
Plugging the expansions \eqref{decdyn}, \eqref{defpetit} into \eqref{eq:VlasovHMFmodel}, we get that $r$ solves
\begin{equation}
\label{eq:r}
\dot m \partial_M \eta + \partial_t r + \{r, h_M\}_{x,v} - \zeta \{ \eta_M, \cos x\}_{x,v} - \eps \zeta \{ r, \cos x\}_{x,v}= 0.
\end{equation}
\subsection{Perturbation equations}

Let us define 
\begin{equation}
\label{deftildeg} \tilde g(t, \theta, a)= r(t, \cdot) \circ \Phi_{M(t)}(\theta, a).
\end{equation}
\begin{Proposition}
For $t \in [0, T^*)$,  there exists  an  hamiltonian $K_{M}(\theta, a)$ analytic in $\mathcal{K}_{M_0}$  such that
$\tilde g(t, \theta, a)$  solves the following equation
\begin{multline}
\label{eqr}
\partial_{t} \tilde g+ \{\tilde g,  H_{M(t)}(a)\}_{\theta, a} - \zeta (t) \{ \Gamma_{M(t)},  \cos x_{M(t)} \}_{\theta, a} - \eps \zeta(t) \{ \tilde g, \cos x_{M(t)}   \}_{\theta, a} \\
=    \eps \dot m \{\tilde g, K_{M(t)} \}_{\theta, a}  -   \dot m \Gamma_M^1,
\end{multline}
where we have set 
\begin{equation}
\label{defGammaM}
\left|
\begin{array}{l}
\Gamma_{M}(a)= G_{M} \circ H_M(a), \\
\Gamma_M^1 (\theta, a)= (\partial_M G_M)\circ H_M  - \cos x_M G_M' \circ H_M.
\end{array}
\right.
\end{equation} 
Moreover,   we can express $\zeta $ as
\begin{equation}
\label{zetapreuve}
\zeta(t)=  \int_{\theta, a}  \tilde g(t, \theta, a)  \cos x_{M(t)}(\theta, a)\, d\theta da
\end{equation}
and the constraint \eqref{orthor} becomes
\begin{equation}
\label{orthor2}
 \int_{\mathbb{T}\times[a_{-}, a_{+}]} \tilde g(t, \theta, a)  C_{0,M(t)} \, d\theta da = 0
\end{equation}
\end{Proposition}
Note that by using \eqref{omegaMdef}, we further have that
$$\{\tilde g , H_{M (t)}(a)\}_{\theta, a} = \omega_{M(t)}(a) \partial_{\theta} \tilde g.$$
We also observe that thanks to \eqref{orthor2}, we have
\begin{equation}
\label{zetapreuve2}
\zeta(t)=  \int_{\theta, a}  \tilde g(t, \theta, a)  \mathcal{C}_{M(t)}(\theta, a)\, d\theta da \end{equation}
where
\begin{equation}
\label{defcostilde}
\mathcal{C}_{M(t)}(\theta, a)=
 \cos x_{M(t)}(\theta, a) - C_{0,M(t)}(a) = \sum_{\ell  \neq 0} C_{\ell,M(t)}(a) e^{i\ell \theta}.
\end{equation}
\begin{proof}
Since the change of variable $(x,v) = \Phi_{M(t)}(\theta, a)$ is symplectic, we deduce from the definition \eqref{deftildeg} and \eqref{eq:r} that
$$
\dot m (\partial_M \eta)\circ \Phi_{M_t} + (\partial_t r) \circ \Phi_{M_t} + \{\tilde g, H_{M_t}\}_{\theta,a} - \zeta \{ \Gamma_{M_t}, \cos x_{M_t}\}_{\theta,a} - \eps \zeta \{ \tilde g, \cos x_{M_t}\}_{\theta,a}= 0.
$$
Next, we observe that since $\eta_M(x,v) = G_M( h_M(x,v))$,  we have
$$ (\partial_M \eta)\circ \Phi_M = (\partial_M G_M) \circ H_M  - \cos x_M G_M' \circ H_M = \Gamma_M^1$$
and that 
\begin{multline}
\label{pert1}
 (\partial_{t}r(t, \cdot) ) \circ \Phi_{M} = \partial_{t} \tilde g  - \dot M (t) \partial_{M} \Phi_{M(t)} \cdot \nabla_{x,v} r (t, \cdot) \circ \Phi_{M}
\\ = \partial_{t} \tilde g  -  \dot M (t) \big[ X_{M(t)} \cdot \nabla_{x,v }r(t, \cdot)\big]\circ \Phi_{M(t)} 
\end{multline}
where  the vector field $X_{M}(x,v)$  is defined for $M \in [M_{0}- \mu, M_{0} + \mu]$ by 
$$ X_{M}( \Phi_{M}(\theta, a))= \partial_{M} \Phi_{M}(\theta, a).$$
We then observe that there exists an analytic $\tilde K_{M}$ defined in $\mathcal{K}_{M_0}$ such that for all function $f$, 
\begin{equation}
\label{hampert} X_{M} \cdot \nabla_{x,v} f = \{f,\tilde K_{M}\}_{x,v}.
\end{equation}
Indeed,  this follows from general symplectic geometry argument.  Since $\Phi_{M}$ is symplectic, we have for every $M \in [M_{0}- \mu, M_{0}+ \mu]$
$$ \Phi_{M}^* ( dx \wedge dv)= d\theta \wedge da, $$ 
taking the derivative with respect to $M$, we get by definition of $X_{M}$ that
$ \Phi_{M}^*\big[ \mathcal{L}_{X_{M}} (dx \wedge dv )\big] = 0$ where $\mathcal{L}_{X}$ stands for the Lie derivative. Therefore we have 
$  \mathcal{L}_{X_{M}} (dx \wedge dv)=0.$  By using the Cartan formula, we then get that
$$0 = \mathcal{L}_{X_{M}} (dx \wedge dv)= d(i_{X_{M}} (dx\wedge dv)) + i_{X_{M}} d( dx \wedge dv)=d(i_{X_{M}} (dx\wedge dv)),$$  
where $i_X$ denote the interior product. 
The one form $i_{X_{M}}( dx \wedge dv) $ is thus closed and thus one can find $\tilde K_{M}$ defined in $\mathcal{K}_{M_0}$
up to a constant so that
$$ i_{X_{M}}( dx \wedge dv)= d \tilde {K}_{M}.$$
This is  equivalent to $\{ \varphi, \tilde K_{M} \}_{x,v}=  X_{M}  \cdot \nabla_{x,v}\varphi$ for every $\varphi$.

Consequently, we can use \eqref{hampert}, \eqref{pert1} and again that $\Phi_{M}$ is symplectic to obtain that
$$  [\partial_{t}r(t, \cdot) ] \circ \Phi_{M} = \partial_{t} \tilde g - \dot M (t) \{ \tilde g, \tilde K_{M} \circ \Phi_{M(t)}\}_{\theta, a}.$$
We then conclude  by setting $K_{M}(\theta, a)= \tilde{K}_{M}\circ \Phi_{M}$.
The expression \eqref{zetapreuve}, just follows from the definition
 $ \zeta(t) = \mathcal{C}[ r]$ and a change of variable. The property  \eqref{orthor2} also follows from \eqref{orthor} and a change of variable.

Note that in the case that $\Phi_{M}$ is given by a generator function $S(x,a, M)$ such that
$ v = \partial_{x} S(x, a, M),$  $\theta= \partial_{a} S(x,a, M)$,  which is the case for the simple pendulum, we can express,  up to a constant,  $K_{M}$
in terms of $S$ by 
$$ K_{M}(\theta, a) = \partial_{M} S(x_{M}(\theta, a), a, M), $$
we refer to \cite{Arnold} for example.  This expression is another way to see the analyticity of $K_{M}$ since the generator function is analytic in
the case of the pendulum.
\end{proof}

We shall then derive a perturbation equation in twisted coordinates along the characteristics of the  modulated  free transport
$ \partial_{t} + \omega_{M(t)} \partial_{\theta}.$
It is convenient to use the  notations
\begin{equation}
\label{defthetat+}
\Theta_{t}^+(\theta, a)= \theta + \int_{0}^t \omega_{M(s)}(a) \, ds, \quad P_{t}^+(\theta, a)= (\Theta_{t}^+, a).
\end{equation}
which corresponds to the characteristics of the  transport operator $\partial_{t}+ \omega_{M(t)} \partial_{\theta}.$
We then set
\begin{equation}
\label{defg} g(t, \theta, a)=  \tilde g\big(t, P_{t}^+(\theta, a) \big)= r\big(t, \Phi_{M(t)}\circ P_{t}^+(\theta, a)\big).
\end{equation}
We shall also use
$$  \Theta_{t}^-(\theta, a)= \theta - \int_{0}^t \omega_{M(s)}(a) \, ds, \quad P_{t}^-(\theta, a)= (\Theta_{t}^-, a)$$
so that
$$ P_{t}^+ \circ P_{t}^-= P_{t}^-\circ P_{t}^+ = \mbox{Id}.$$

Since the change of variable $(\theta, a) \mapsto P_{t}^+$ is symplectic, we get that $g$ solves
\begin{multline}
\label{eqg}
\partial_{t} g  - \zeta (t)  \{ \Gamma_{M(t)}, \cos x_{M(t)} \circ P_{t}^+ \}_{\theta, a} - \eps \zeta (t)   \{ g, \cos x_{M(t)} \circ P_{t}^+ \}_{\theta, a} 
\\ =  \eps \dot m \{ g, K_{M(t)} \circ P_{t}^+ \}_{\theta, a}
  - \dot m  \Gamma_{M(t)}^1 \circ P_t^+.
\end{multline}
In terms of $g$, by a  change of variable,  the expression \eqref{zetapreuve} becomes
\begin{equation}
\label{zetapreuve3}
\zeta(t)= \int_{\theta, a} g(t, \theta, a) \mathcal{C}_{M(t)}\circ P_{t}^+ \, d\theta da
\end{equation}
and \eqref{orthor2}
\begin{equation}
\label{orthog}
\int_{\theta, a} g(t, \theta, a) C_{0,M(t)}(a)\, d\theta da= 0.
\end{equation}
\subsection{Set up of  the a  priori estimates}
We consider for $t \in [0, T^*),$ the quantity
\begin{equation}
\label{Qdef}
Q_{\sigma} (T) =\sup_{t \in [0, T]} \left(  \langle t \rangle^{\sigma-1} \zeta(t) +  \Norm{ g(t)}{H^{\sigma-4}} + { 1 \over \langle t \rangle^3} \Norm{g }{H^\sigma} \right).
\end{equation}
Note that for $t \in [0, T^*),$ $r$ is supported in $\mathcal{K}_{M_0}$ so $g$ is supported in $\mathbb{T} \times [a_{-}, a_{+}]$
where the map $\Phi_{M}$ is analytic. We thus deduce from \eqref{petitinit} that we have initially
\begin{equation}
\label{ginit}
 \Norm{g^0 }{H^{\sigma+1}} \leq C_{2}\mathcal{R}_0
\end{equation}
for some $C_{2}>0$.

For some $\mathcal{R}$ sufficiently large to be chosen, we can thus consider
\begin{equation}
\label{Tepsdef}
T^\eps = \sup\{ T \in (0,  T^*), \, Q_{\sigma}(t) \leq \mathcal{R}, \, \forall t \in [0, T]\}.
\end{equation}

\subsection{Estimate of $\dot m$}
Before turning to the estimate of $Q_{\sigma}$, we shall use the constraint \eqref{orthog} to estimate $\dot m$.

\begin{Proposition}
\label{propdotm}
For every $\mathcal{R} \geq 1$ in \eqref{Tepsdef}, there exists  $\eps_0>0$ such that for every $\eps \in (0, \eps_0]$, we have
 for every  $ t \in [0, T^\eps)$, 
 $$ |\dot m (t)  | \leq  C \mathcal{R} \eps {| \zeta (t)| \over \langle t \rangle^{\sigma - 4 }}
  \leq  C \mathcal{R}^2 \eps { 1 \over \langle t \rangle^{ 2 \sigma - 5 }}.$$

\end{Proposition}

Note that as a consequence of the above estimate, by setting $M_\eps = M_0 + \eps  m(T_\eps)$, we get that
\begin{equation}
\label{Mlimite}
|M_\eps - M(t) |\lesssim { \mathcal{R}^2 \eps^2 \over \langle t\rangle^{2 \sigma - 6}}.
\end{equation}
\begin{proof}
By multiplying \eqref{eqg} by $C_{0,M(t)}(a)$, integrating in $\theta, a$ and using the orthogonality condition
 \eqref{orthog}, we get that
 \begin{multline}
 \label{eqdotm} \dot m \int_{\theta, a}\Gamma_{M(t)}^1 \circ P_t^+ \, C_{0,M(t)}(a)\, d\theta da 
 =- \eps \dot m \int_{\theta, a}  \{ g, K_{M(t)} \circ P_{t}^+ \}_{\theta, a} C_{0,M(t)} \, d\theta da \\
-  \eps \zeta(t)   \int_{\theta, a}   \{ g, \cos x_{M(t)} \circ P_{t}^+ \}_{\theta, a}  C_{0,M(t)} \, d\theta da
  - \eps \dot m \int_{\theta, a} g \, \partial_{M} C_{0,M(t)} \, d\theta da,  
  \end{multline}
  where we have used that by  change of variable and integration by parts in $\theta$, 
$$  \int_{\theta, a}  \{ \Gamma_{M(t)}, \cos x_{M(t)} \circ P_{t}^+ \}_{\theta, a} C_{0,M(t)}(a)\, d\theta da=
     \int_{\theta, a}  \{ \Gamma_{M(t)}, \cos x_{M(t)} \}_{\theta, a} C_{0,M(t)}(a)\, d\theta da=0. $$
Next, we observe that thanks to \eqref{defGammaM}
\begin{multline*}
 \int_{\theta, a}\Gamma_{M(t)}^1 \circ P_t^+ \, C_{0,M(t)}(a)\,  d\theta da
 =  \int_{a}  (\partial_MG_M)\circ H_M C_{0,M} \, da - \int_{\theta, a}  \cos x_M \circ  P_t^+ G_M'(H_M) C_{0,M} \, d\theta da
  \\ = \int_{a} (\partial_M G_M)\circ H_M C_{0,M} \, da - \int_{\theta, a} G_M'(H_M) C_{0,M}^2 \, d\theta da.
\end{multline*}
This is exactly the expression in  \eqref{FM1}, taken at $M$, we thus  further get that
$$  \int_{\theta, a}\Gamma_{M(t)}^1 \circ P_t^+ \, C_{0,M(t)}(a)\,  d\theta da = 1 + \mathcal{Q}_{\mathcal{C}, M(t)}(0)$$
which does not vanish thanks to (H3), and where $ Q_{\Cc,M}(t):= Q_{\mathcal{C}}[\Psi_{\delta(\overline{M}_0)}(k^2)](t)$ is defined in \eqref{QC2}.
This yields that  for  $t \leq T^\eps$, we have
\begin{equation}
\label{dotm4}
\Big |  \int_{\theta, a}\Gamma_{M(t)}^1 \circ P_t^+ \, C_{0,M(t)}(a)\,  d\theta da
 \Big| \geq c_0 >0.
 \end{equation}
 Next, we  estimate the right-hand side of  \eqref{eqdotm}. We first have
 \begin{equation}
 \label{dotm3}
 \Big |\eps \dot m \int_{\theta, a} g \, \partial_{M} C_{0,M(t)} \, d\theta da \Big| \leq \eps | \dot m | \Norm{ g }{L^2}
  \leq  \eps | \dot m | \mathcal{R}.
  \end{equation}
  Next, we observe that
  \begin{multline*}  \int_{\theta, a}   \{ g, \cos x_{M(t)} \circ P_{t}^+ \}_{\theta, a}  C_{0,M(t)} \, d\theta da
  =  \int_{\theta, a}   g \{ \cos x_{M(t)} \circ P_{t}^+ ,  C_{0,M(t)}  \}_{\theta, a} d\theta da  \\
  =  \int_{\theta, a}   g \{ \cos x_{M(t)} \ ,  C_{0,M(t)}  \}_{\theta, a}\circ P_{t}^+ d\theta da = 
   \int_{\theta, a}   g\,  \varphi \circ P_t^+ \, d\theta da$$
   \end{multline*}
   where we have set
   $$ \varphi (\theta, a)= \big\{\cos x_{M(t)} \ ,  C_{0,M(t)}  \big\}_{\theta, a} = \partial_\theta  \cos x_{M(t)}   C_{0,M(t)}'.$$
   Since $\int_\theta \varphi(\theta, a) \, d\theta = 0$, we thus get from Corollary \ref{corchiant}
   that
   \begin{equation}
   \label{dotm2}
   \Big |   \int_{\theta, a}   \{ g, \cos x_{M(t)} \circ P_{t}^+ \}_{\theta, a}  C_{0,M(t)} \, d\theta da \Big|
   \lesssim { 1 \over \langle t \rangle^{\sigma - 4}  } \Norm{g(t) }{H^{\sigma -4}} \lesssim {\mathcal{R} \over  \langle t \rangle^{\sigma - 4}  }.
   \end{equation}
   Finally, since
   $$  \int_{\theta, a}  \{ g, K_{M(t)} \circ P_{t}^+ \}_{\theta, a} C_{0,M(t)} \, d\theta da
   = \int_{\theta, a} g\,  \{  K_{M(t)} ,   C_{0,M(t)}  \}_{\theta, a} \circ P_{t}^+\, d\theta da, $$
   we have
   \begin{equation}
   \label{dotm1}
    \Big| \int_{\theta, a}  \{ g, K_{M(t)} \circ P_{t}^+ \}_{\theta, a} C_{0,M(t)} \, d\theta da\Big| \lesssim \mathcal{R}.
    \end{equation}
    We thus get from \eqref{eqdotm} and \eqref{dotm4}, \eqref{dotm3}, \eqref{dotm2}, \eqref{dotm1}, that
    $$ | \dot m | \lesssim  \eps \mathcal{R} |\dot m | +  \eps \mathcal{R }  {| \zeta (t)| \over \langle t \rangle^{\sigma - 4 }}.$$ 
    We thus conclude assuming that $\eps \mathcal{R}$ is sufficiently small.

\end{proof}

\subsection{Evolution of  $\zeta$}
By integrating in time \eqref{eqg} and by using the definition of $\zeta$ in \eqref{zetapreuve3}, we get that
\begin{equation}
\label{eqzetaNL}
 \zeta(t) = \int_0^t  \zeta(s)\int_{\theta, a} \Gamma_{M(s)} \{ \cos x_{M(s)} \circ P_s^+, \mathcal{C}_{M(t)}\circ P_t^+ \} d \theta da \, ds  +
\mathcal{F}(t)
\end{equation}
where 
\begin{multline}
\label{sourceFzeta}
\mathcal{F}(t)= 
- \eps   \int_0^t  \zeta(s)\int_{\theta, a}  g(s) \{ \cos x_{M(s)} \circ P_s^+, \mathcal{C}_{M(t)}\circ P_t^+\} d \theta da \, ds  \\
+\eps   \int_0^t  \dot m (s)\int_{\theta, a}  g(s) \{ K_{M(s)} \circ P_s^+,\mathcal{C}_{M(t)} \circ P_t^+ \} d \theta da  \, ds
- \int_0^t \dot m \int_{\theta, a} \{ \Gamma_{M(s)}^1 \circ P^+_s, \mathcal{C}_{M(t)}\circ P_t^+ \} d\theta da \, ds \\
+ \int_{\theta, a} g^0(\theta, a) \mathcal{C}_{M(t)}\circ P_t^+ \, d\theta da.
\end{multline}
Note that  since $\Gamma_M$ depends only on $a$ the first term in the right-hand side of \eqref{eqzetaNL} can be written 
\begin{multline*}
\int_0^t  \zeta(s)\int_{\theta, a} \Gamma_{M(s)} \{ \cos x_{M(s)} \circ P_s^+, \mathcal{C}_{M(t)}\circ P_t^+ \} d \theta da \, ds 
\\= \int_0^t  \zeta(s)\int_{\theta, a} \Gamma_{M(s)} \{ \cos x_{M(s)}, \mathcal{C}_{M(t)}\circ (P_t^+  P_s^-)\} d \theta da \, ds  
= \int_0^t  \zeta(s) \mathcal{G} (t,s) \, ds
\end{multline*}
where the kernel $\mathcal{G}$ is given by 
\begin{equation}
\label{kernelGdef}
\mathcal{G}(t,s) =\int_{\theta, a} \Gamma_{M(s)} \{ \cos x_{M(s)}, \mathcal{C}_{M(t)}\circ (P_t^+  P_s^-)\} d \theta da. 
\end{equation}
We can thus rewrite \eqref{eqzetaNL} as
\begin{equation}
\label{eqzetaNL2}
\zeta(t) = \int_0^t  \mathcal{G}(t,s) \zeta(s) + \mathcal{F}(t). 
\end{equation}
\subsection{Linear estimates}
In this section, we shall consider the equation \eqref{eqzetaNL2} for  any source term $\mathcal{F}$.
The goal of this subsection is to prove the following:
\begin{Proposition}
\label{proplineaire}
Consider the equation \eqref{eqzetaNL2}.
 For every $L \geq 0$, there exist $C>0$ and $\delta_0>0$ such that for every $\eps>0$  and  every $\mathcal{R}\geq 1$,  such that $\eps \mathcal{R}\leq \delta_0$, we have for  every $T \in [0, T^\eps)$ the estimate
 \begin{equation}
 \label{estproplin}
  \sup_{t \in  [0, T]} \langle t\rangle^L |\zeta(t)| \leq C   \sup_{t \in [0, T]} \langle  t\rangle^L  |\mathcal{F}(t)| 
 \end{equation}
\end{Proposition}
As a starting point, 
by using $M_\eps$ defined before \eqref{Mlimite}, we can  expand $\mathcal{G}$ defined in \eqref{kernelGdef} as
\begin{equation}
\label{expandG}
\mathcal{G}(t,s) = \mathcal{G}_\infty(t-s) + \mathcal{G}_r^1(t,s) + \mathcal{G}_r^2(t,s)
\end{equation}
where
\begin{align}
\label{Ginfty}
&  \mathcal{G}_\infty(t-s) = \int_{\theta, a} \Gamma_{M_\eps} \{ \cos x_{M_\eps}, \mathcal{C}_{M_\eps}\circ P^\eps_{t-s}\} d \theta da=  \int_{\theta, a} \Gamma_{M_\eps} \{ \cos x_{M_\eps}, \cos x_{M_\eps} \circ P^\eps_{t-s}\} d \theta da  \, \\
\label{defGr1}&  \mathcal{G}_r^1 (t,s)  =  \int_{\theta, a} (\Gamma_{M(s)}- \Gamma_{M_\eps}) \{ \cos x_{M(s)}, \mathcal{C}_{M(t)}\circ (P_t^+  P_s^-)\} d \theta da
  \\
  & \quad \quad \quad \quad  +\int_{\theta, a}   \Gamma_{M_\eps} \{ \cos x_{M(s)} - \cos x_{M_\eps}, \mathcal{C}_{M(t)}\circ (P_t^+  P_s^-)\} d\theta da \nonumber
  \\
  & \quad \quad \quad \quad +  \int_{\theta, a} \Gamma_{M_\eps} \{ \cos x_{M_\eps}, (\mathcal{C}_{M(t)} - \mathcal{C}_{M_\eps})\circ (P_t^+  P_s^-)\} 
\, d \theta da  \nonumber \\
\label{defGr2}&   \mathcal{G}_r^2 (t,s)  =  \int_{\theta, a} \Gamma_{M_\eps} \{ \cos x_{M_\eps}, \mathcal{C}_{M_\eps}\circ (P_t^+  P_s^-) -  \mathcal{C}_{M_\eps}\circ 
 P^\eps_{t-s}\}  \, d\theta da
\end{align}
and where we have set
\begin{equation}
\label{Pepsded}
P^\eps_\tau(\theta, a) =( \theta + \tau \omega_{M_\eps} (a), a).
\end{equation} 
In order to prove Proposition \ref{proplineaire}, we shall first  study the linear Volterra equation
\begin{equation}
\label{eqzetainfty}
\zeta(t) =\int_0^t \mathcal{G}_\infty(t-s) \zeta(s) \, ds + \mathcal{S}(t), \quad t\geq 0
\end{equation}
for a given source term $\mathcal{S}$ and then handle \eqref{eqzetaNL2} perturbatively.
The main result for \eqref{eqzetainfty} is the following:
\begin{Lemma}
\label{Lemmavolterra}
For every $\mathcal{S} \in L^\infty(\mathbb{R}_+)$, there exists a unique
 solution $\zeta  \in L^{\infty}(\mathbb{R}_+)$ of \eqref{eqzetainfty}. Moreover, for every $L>0$,  there exists $C>0$ such that for every $\eps \in (0, \eps_0]$ 
 and $T\geq 0$,  we have
 \begin{equation}
 \label{decaylineaire}
 \sup_{ t \in [0, T]} \langle t \rangle^L | \zeta(t) | \leq C  \sup_{ t \in [0, T]} \langle t \rangle^L | \mathcal{S}(t) |.
 \end{equation} 
\end{Lemma}
\begin{proof}
We first observe that from  \eqref{Ginfty}, by using the symplectic change of variable $(x,v) = \Phi_{M_\eps}(\theta, a)$, we obtain that
$$  \mathcal{G}_\infty(t) 
= \int_{x,v} \eta_{M_\eps}(x,v) \{ \cos x, \cos X \circ \varphi^t_{M_\eps} \}_{x,v}\, dx dv =
 K_{\mathcal{C},  M_\eps}(t) $$
where $\varphi^t_{M_\eps}$ is the flow of the Hamiltonian $h_{M_\eps}$.
The first part of the statement then follows from  the analysis of \cite{Faou2021-qm} Proposition 3.1, Lemma 3.2 and Corollary 3.3 i) for example.
Indeed, for a smooth stationary state, we always have that $K_{\mathcal{C}, M_\eps} \in L^1(\mathbb{R}_+)$ and thus the existence and uniqueness of a bounded solution which satisfies
\begin{equation}
\label{zetalin1}
 \Norm{ \zeta }{L^\infty} \leq C \Norm{\mathcal{S}}{L^\infty}
 \end{equation}
 follows from the Penrose condition and general properties of Volterra equation.

To get \eqref{decaylineaire}, we can further use that $\nabla \eta_{M_\eps}$ is compactly supported in $\mathcal{K}_{M_0}$.
 Indeed by using the expansion \eqref{scaledcos}, we can write
 $$  \mathcal{G}_\infty(t)= \sum_{k\neq 0} ik \int_{a}  |C_{k, M_\eps}(a)|^2 \Gamma'(a)  e^{-ik t \omega_{M_\eps}(a)} \, \dd a$$
 By using Lemma \ref{time-decay-ipp} we thus get that
 $$ \Big |  \int_{a}  |C_{k, M_\eps}(a)|^2 \Gamma'(a)  e^{-ik t \omega_{M_\eps}(a)} \, \dd a \Big |
  \lesssim  { 1 \over \langle t \rangle^B} { 1 \over |k|^B} \Norm{ C_k }{H^{B}(a_-, a_+)}^2$$
  and thus by the Bessel-Parseval identity that
  $$ |\mathcal{G}_\infty(t)| \lesssim_B { 1 \over  \langle t \rangle^B}$$
  for every $B\geq 1$.
  This allows to easily get \eqref{decaylineaire} by induction. Indeed, for $L\geq 1$, we can write
  $$ t^L \zeta(t)= \int_0^t \mathcal{G}_\infty(t-s)  s^L\zeta(s) \, \dd s + \mathcal{S}_L(t), \,  \mathcal{S}_L(t)=  \int_0^t   \mathcal{G}_\infty(t-s)(t^L - s^L) \zeta(s)\, 
  \dd s + t^L \mathcal{S}(t).$$
  We thus get
  $$ \Norm{t^L \zeta }{L^\infty}
   \lesssim  \Norm{t^L \mathcal{S} }{L^\infty} + \sup_{t\geq 0} \Big| \int_0^t   \mathcal{G}_\infty(t-s)(t^L - s^L) \zeta(s)  \dd s\Big|.$$
   To conclude, we observe that
   $$ |\mathcal{G}_\infty(t-s)(t^L - s^L)|\lesssim { 1 \over \langle t-s \rangle^{B-1}}( (t-s)^{L-1} + s^{L-1})$$
   and thus, for $B$ sufficiently large, we have
   $$  \sup_{t\geq 0} \Big| \int_0^t   \mathcal{G}_\infty(t-s)(t^L - s^L) \zeta(s)  \dd s\Big| \lesssim \Norm{ \langle t\rangle^{L-1} \zeta }{L^\infty}.$$
\end{proof}

We are now in position to prove Proposition \ref{proplineaire}.
\begin{proof}[Proof. of Proposition \ref{proplineaire}]
By using Lemma \ref{Lemmavolterra}, we obtain that
\begin{equation}
\label{finlin1} \sup_{t \in [0, T]} \langle t\rangle^L |\zeta(t) | \lesssim  \sup_{t \in [0, T]} \langle  t\rangle^L  |\mathcal{S}(t)|
\end{equation}
where
\begin{equation}
\label{finlin2} \mathcal{S}(t)= \mathcal{F}(t) +  \int_0^t \mathcal{G}_r^1(t, s) \zeta (s) \, \dd s +  \int_0^t \mathcal{G}_r^2(t, s) \zeta (s) \, \dd s
\end{equation}
and $\mathcal{G}_r^{i}, \,i=1, \, 2$ are defined in \eqref{defGr1}, \eqref{defGr2}.

To conclude, we shall thus  study the kernels $\mathcal{G}_r^i$, $i=1,\, 2$.

It is then  convenient for any three functions $F, G, H$ depending only on $a$ and for  $l\neq 0$ to use the notation
\begin{equation}
\label{trilin}
 \mathcal{T}_{\ell, s, t}[F, G, H](a)= F(a) \Big( -i\ell ( G(a) H'(a) + G'(a) H(a)) + \ell^2 G(a) H(a) \int_s^t \omega_{M(\tau)}'(a)  \, \dd \tau  \Big).
 \end{equation}

By using again the expansion \eqref{scaledcos}, the above definition \eqref{trilin} and $M_\eps$ defined in \eqref{Mlimite},  we get that
$$ \mathcal{G}_r^1(t, s) \zeta (s) =  \mathcal{G}_r^{1, 1}(t, s) +   \mathcal{G}_r^{1, 2}(t, s) +  \mathcal{G}_r^{1, 3}(t, s)$$
where using \eqref{IAOS}
$$\mathcal{G}_r^{1, 1}(t, s)= 
   \sum_{\ell \neq 0 }  \int_a  \mathcal{T}_{\ell,s,t}[ \Gamma_{M(s)} - \Gamma_{M_\eps}, C_{\ell, M(s)}, C_{\ell, M(t)}]  e^{i\ell \int_s^t \omega_{M(\tau)}}
 \, \dd a 
  =  \sum_{\ell \neq 0 }   \, I[\mathcal{A}_{1, \ell}, \omega_{M_\eps} ] \big(\ell (t-s)\big), 
$$
$$ \mathcal{A}_{1, \ell}=  \mathcal{T}_{\ell,s,t}[ \Gamma_{M(s)} - \Gamma_{M_\eps}, C_{\ell, M(s)}, C_{\ell, M(t)}]  e^{i\ell \int_s^t (\omega_{M(\tau)}- \omega_{M_\eps} )  \dd \tau} ,$$
$$
\mathcal{G}_r^{1, 2}(t, s)=  \sum_{\ell \neq 0 }  \ell \int_a   \mathcal{T}_{\ell,s,t}[\Gamma_{M_\eps}, C_{\ell, M(s)}- C_{\ell, M_\eps}, C_{\ell, M(t)} ]e^{i\ell \int_s^t (\omega_{M(\tau) } \dd \tau} \, \dd a  \\
  =   \sum_{\ell \neq 0 }   \, I[\mathcal{A}_{2, \ell}, \omega_{M_\eps} ] \big(\ell (t-s)\big), 
$$
$$
 \mathcal{A}_{2, \ell}=\mathcal{T}_{\ell,s,t}[\Gamma_{M_\eps}, C_{l, M(s)}- C_{\ell, M_\eps}, C_{\ell, M(t)} ]
 e^{i\ell \int_s^t (\omega_{M(\tau)}- \omega_{M_\eps} ) \dd \tau } , 
$$
$$
\mathcal{G}_r^{1, 3}(t, s)=  \sum_{\ell \neq 0 }  \ell \int_a   \mathcal{T}_{\ell,s,t}[\Gamma_{M_\eps},  C_{\ell, M_\eps}, C_{\ell, M(t)}- C_{\ell, M_\eps}  ]e^{i\ell \int_s^t (\omega_{M(\tau) } \dd \tau} \, \dd a 
  =   \sum_{\ell \neq 0 }   \, I[\mathcal{A}_{3, \ell}, \omega_{M_\eps}] \big(\ell (t-s)\big), 
$$
$$
 \mathcal{A}_{3, \ell}=\mathcal{T}_{\ell,s,t}  [\Gamma_{M_\eps},  C_{\ell, M_\eps}, C_{\ell, M(t)}- C_{\ell, M_\eps}  ]
 e^{i\ell \int_s^t (\omega_{M(\tau)}- \omega_{M_\eps} ) \dd \tau }. 
$$

 We can thus use again Lemma \ref{time-decay-ipp} to estimate the kernels $\mathcal{G}_r^{1, i}$.
 Note that by the smoothness of $\cos x_M$  with respect to $(M, \theta, a)$ in  $[M_0- \mu, M_0+\mu] \times \mathbb{T} \times [a_-, a_+]$, 
 which contains the support of $\Gamma_M$ for every $M\in [M_0- \mu, M_0+\mu]$, 
 we have
 $$ \sup_{M \in [M_0- \mu, M_0+\mu]}  \Norm{ C_{\ell,M} }{W^{B, \infty}([a_-, a_+])}   +\Norm{\partial_M  C_{\ell,M}}{W^{B, \infty}([a_-, a_+])}
  \lesssim_{B,K} { 1 \over \langle \ell\rangle^K}$$
  for every $B$ and $K$.
  
  Moreover, thanks to the smoothness of $\omega$ with respect to $(M, a)$ in   $[M_0- \mu, M_0+\mu] \times [a_-, a_+]$, we have
  $$\sup_{M \in [M_0- \mu, M_0+\mu]} \left\| \int_s^t \omega_{M(\tau)} \, \dd \tau \right \|_{W^{B, \infty}([a_-, a_+])}
   \lesssim_B |t-s|, $$ 
   and by using in addition \eqref{Mlimite}, we also have that
   $$  \sup_{0\leq s \leq t \leq T} \left\| e^{i\ell \int_s^t (\omega_{M(\tau) } - \omega_{M_\eps})\dd \tau} \right\|_{W^{B, \infty}([a_-, a_+])}
   \lesssim_{B} \langle  l \rangle^B $$
   where this last estimate  does not involve  $T$ in the right hand side.
  
  By combining these estimates, we easily obtain that  
  $$ \sum_{j=1}^3 \Norm{\mathcal{A}_{j,\ell}}{W^{B, \infty}([a_-, a_+])}  \lesssim_{B, K}  { \langle t-s \rangle \over \langle \ell\rangle^K}
  \sup_{t \in [0, T^\eps]} |M(t) - M_\eps | \lesssim _{B, K}  { \langle t-s \rangle \over \langle \ell\rangle^K} \mathcal{R}^2 \eps^2 $$
  where the last estimate relies again on \eqref{Mlimite}.
  This yields thanks to Lemma \ref{time-decay-ipp} and by choosing $K$ sufficiently large ($K\geq B + 3$)
 that
  $$ |\mathcal{G}_r^{1, i}(t,s) |
   \lesssim_B  \mathcal{R}^2 \eps^2  { 1 \over \langle t-s \rangle^{B-1} }, \quad i=1, \, 2, \, 3$$
   for any $B \geq 1$
   which finally gives 
   \begin{equation}
   \label{estGr1}
    |\mathcal{G}_r^1(t,s) |
   \lesssim_B \mathcal{R}^2 \eps^2  { 1 \over \langle t-s \rangle^{B-1} }.
   \end{equation}

It remains $\mathcal{G}_r^2.$ From similar computations as above,  we get from \eqref{defGr2} that
 \begin{multline*}
 \mathcal{G}_r^2(t,s)= 
  \sum_{\ell\neq 0}  \int_a \Gamma_{M_\eps}'(a) i\ell |C_{\ell, M_\eps}|^2 \left( e^{i\ell (\int_s^t (\omega_{M(\tau)} -\omega_{M_\eps} )\, \dd \tau} - 1 \right)
   e^{i\ell (t-s) \omega_{M_\eps}} \,  \dd a \\
   = \sum_{\ell\neq 0} I[\mathcal{A}_{4,\ell}, \omega_{M_\eps}]\big(\ell(t-s)\big)
 \end{multline*}
 where
 $$ \mathcal{A}_{4,l}= \Gamma_{M_\eps}'(a) i\ell |C_{\ell, M_\eps}|^2 \left( e^{i\ell (\int_s^t ( \omega_{M(\tau)} -\omega_{M_\eps} )  \, \dd \tau} - 1 \right).$$
 We shall use again Lemma \ref{time-decay-ipp}.
 To estimate the amplitude,  we observe that thanks to \eqref{Mlimite} and again the  regularity  of $\omega_M$ with respect to $M$, we have that 
  for every $\alpha \geq 0$, 
 $$ \Big|\int_s^t  (\partial_a^\alpha \omega_{M(\tau)}  -  \partial_a^\alpha \omega_{M_\eps}) \, \dd \tau \Big|  \lesssim_\alpha  \mathcal{R}^2 \eps^2 \int_s^t { 1 \over \langle \tau\rangle^{2 \sigma -3} }\lesssim  \mathcal{R}^2 \eps^2.$$
 assuming that $2\sigma  > 4$.
 Therefore we have that
 $$ \Big\| e^{i\ell (\int_s^t \omega_{M(\tau)} \, \dd \tau - (t-s) \omega_{M_\eps}} - 1 \Big\|_{W^{B, \infty}([a_-, a_+])} 
 \lesssim \langle \ell\rangle^{B}   \mathcal{R}^2 \eps^2.$$
 This allows to obtain from Lemma \ref{time-decay-ipp} by using  similar arguments as above   that
 $$ \left|I[\mathcal{A}_{4,\ell}, \omega_{M_\eps}]\big(\ell(t-s)\big)\right| \lesssim { 1 \over \langle l \rangle^2} { 1 \over \langle t-s \rangle^{B}}$$
 and thus  that
   \begin{equation}
   \label{estGr2}
    |\mathcal{G}_r^2(t,s) |
   \lesssim_B \mathcal{R}^2 \eps^2  { 1 \over \langle t-s \rangle^{B} }.
   \end{equation}
   
   We thus get from \eqref{finlin1}, \eqref{finlin2} and \eqref{estGr1}, \eqref{estGr2} that
$$  \sup_{t \in [0, T]} \langle t\rangle^L |\zeta(t)| \lesssim  \sup_{t \in [0, T]} \langle  t\rangle^L  |\mathcal{F}(t)| 
 +  \mathcal{R}^2 \eps^2 \sup_{t \in [0, T]}  \left( \langle  t\rangle^L \int_0^t { 1 \over \langle t-s \rangle^{B-1} } { 1 \over \langle s \rangle^{L} }  \, \dd s \right)   \sup_{t \in [0, T]} \langle t\rangle^L |\zeta(t)|.$$
 Since we have for $B$ sufficiently large ($B\geq L+2$)   that
 $$ \langle t\rangle^L \int_0^t { 1 \over \langle t-s \rangle^{B-1} } { 1 \over \langle s \rangle^{L} }  \, \dd s  \lesssim 1, $$
 we can get \eqref{estproplin} assuming that $\mathcal{R} \eps$ is sufficiently small.
 This ends the proof
\end{proof}

\subsection{Estimate of $\zeta$}
We can now consider \eqref{eqzetaNL2} with $\mathcal{F}$ given by \eqref{sourceFzeta}. Our goal is to prove:
\begin{Proposition}
\label{propzetaNL}
There exists $C_0>0$ and $C>0$  such that for every  $\eps \in (0, \eps_0],$ every  $ T \in [0, T^\eps)$ such that
\begin{equation}
\label{Tcontrainte} C_0 \langle T\rangle  \mathcal{R}^2 \eps  \leq { 1\over 2}
\end{equation} 
we have
\begin{equation}
\label{lestimdezeta}
\mathcal{M}_{T, \sigma -1} (\zeta)   \leq   C( \mathcal{R}_0 + \mathcal{R}^2 \eps), 
\end{equation}
where $\mathcal{M}_{T, s} = \sup_{t \in [0, T]} \langle t \rangle^s | \zeta (t) |.$
 \end{Proposition}
\begin{proof}
From Proposition \ref{proplineaire}, we get that
\begin{equation}
\label{proofzetaNL1} \mathcal{M}_{T, \sigma -1}(\zeta) \lesssim  \mathcal{M}_{T, \sigma -1}(\mathcal{F}),
\end{equation}
we thus only need to estimate the right hand side.

By using \eqref{sourceFzeta}, we can split $\mathcal{F}$ into 
\begin{equation}
\label{FNLsplit}
\mathcal{F}= \mathcal{F}_1+ \mathcal{F}_2 + \mathcal{F}_m + \mathcal{F}_I
\end{equation}
where
\begin{align}
\label{FNLI}
&\mathcal{F}_I =  \int_{\theta, a} g^0(\theta, a) \mathcal{C}_{M(t)}\circ P_t^+ \, d\theta da, \\
\label{FNLm}
&\mathcal{F}_m  = - \int_0^t \dot m \int_{\theta, a} \{ \Gamma_{M(s)}^1 \circ P^+_s, \mathcal{C}_{M(t)}\circ P_t^+ \} d\theta da \, ds, \\
\label{FNL1}
&\mathcal{F}_1 =  - \eps   \int_0^t  \zeta(s)\int_{\theta, a}  g(s) \{ \cos x_{M(s)} \circ P_s^+, \mathcal{C}_{M(t)}\circ P_t^+\} d \theta da \, ds, \\
\label{FNL2}
&\mathcal{F}_2  = \eps   \int_0^t  \dot m (s)\int_{\theta, a}  g(s) \{ K_{M(s)} \circ P_s^+,\mathcal{C}_{M(t)} \circ P_t^+ \} d \theta da  \, ds.
\end{align}
By expanding into  Fourier series,  we have
$$ \mathcal{F}_I(t)  = \sum_{\ell\neq 0} \int_a g_\ell^0(t, a) C_{\ell, M(t)} e^{i\ell \int_0^t  \omega_{M(\tau)}\dd \tau } \, \dd a=
\sum_{\ell\neq 0} I [ g_\ell^0 C_{\ell,M} e^{i\ell \int_0^t  (\omega_{M(\tau)}-\omega_{M_\eps}) \, \dd \tau} ,   \omega_{M_\eps} ] (\ell t).$$
We can again use  Lemma \ref{time-decay-ipp}    to get by the same arguments as above  that for $\ell \neq 0$,
$$\left|  I [ g_\ell^0 C_{\ell,M} e^{i\ell \int_0^t  (\omega_{M(\tau)}-\omega_{M_\eps}) \, \dd \tau},  \omega_{M_\eps}] (\ell t) \right| \lesssim  { 1 \over \langle t \rangle^{\sigma-1}} \Norm{g_{\ell}^0}{H^{\sigma-1}} { 1\over \langle \ell\rangle}$$
and hence, we obtain by Cauchy-Schwarz that 
\begin{equation}
\label{estFI}
 \mathcal{M}_{T, \sigma -1}(\mathcal{F}_I) \lesssim  \Norm{g^0 }{H^{\sigma-1}}\lesssim  \mathcal{R}_0.
 \end{equation}
 In a similar way, by denoting $D_{\ell, M}(a)$ the Fourier coefficients in $\theta$   of the smooth, real  and compactly supported in $[a_-, a_+]$  function $\Gamma^1_M(\theta, a)$, we can 
 write by using again the notation \eqref{trilin} that 
 $$
 \mathcal{F}_m(t) = \int_0^t \dot m(s) \sum_{\ell\neq 0}  \mathcal{T}_{\ell, s, t}[1, D_{\ell,M(s)}, C_{\ell, M(t)}]  e^{i\ell \int_s^t \omega_{M(\tau)} \, \dd \tau}\, \dd s
  = \int_0^t \dot m(s) \sum_{\ell\neq 0} I [\mathcal{A}_\ell,  \omega_{M_\eps}]\big(\ell (t-s)\big)
 $$
where we have set
$$ \mathcal{A}_\ell=   \mathcal{T}_{\ell, s, t}[1, D_{\ell,M(s)}, C_{\ell, M(t)}]  e^{ i\ell  \int_s^t (\omega_{M(\tau)}- \omega_{M_\eps}) \, \dd \tau}. $$
 We can thus use again Lemma \ref{time-decay-ipp}  to get
 $$\left| I [\mathcal{A}_\ell, \omega_{M_\eps}]\big(\ell (t-s)\big)\right| \lesssim_B { 1 \over \langle \ell \rangle^2} { 1 \over \langle t-s \rangle^B}$$
 with $B$ arbitrarily large and hence that
 $$  |\mathcal{F}_m(t)|\lesssim \int_0^t | \dot m (s) |  { 1\over \langle t-s \rangle^{B} }\, \dd s\lesssim
  \mathcal{R}^2 \eps \int_0^t { 1 \over \langle s \rangle^{2\sigma -5 }} { 1 \over \langle t-s \rangle^{B}} \, \dd s$$
  where we have used  Proposition \ref{propdotm} for the last estimate.
  By taking $B$ sufficiently large, this yields
  \begin{equation}
\label{estFm}
 \mathcal{M}_{T, \sigma -1}(\mathcal{F}_m) \lesssim  \mathcal{R}^2 \eps.
 \end{equation}
 It remains to estimate   $\mathcal{F}_i$, $i=1, \, 2$.
 
 These are the crucial terms where echoes arise. We start with $\mathcal{F}_1$ which is the most difficult term.
 By expanding again in Fourier series, we can write
 $$
 \mathcal{F}_1(t)
 =- \eps \int_0^t \zeta(s)   \sum_{\ell\neq 0, \, k}  \int_{a} g_{k+\ell}(s,a)  \mathcal{A}_{k,\ell}^0 e^{i\ell \int_0^t \omega_{M} + ik \int_0^s\omega_M}\, \dd a \dd s
 $$ 
 where as usual $g_n$ stands for the Fourier coefficients of $g$ in $\theta$ and we have set
 $$ \mathcal{A}_{k,\ell}^0(t,s,a)
 =  ik C_{M(s), k} C_{\ell,M(t)}' - il C_{k,M(s)}' C_{\ell,M(t)} - k\ell C_{k,M(s)} C_{\ell,M(t)} \int_s^t \omega'_{M(\tau)}.$$
 By using $M_\eps$ defined just before \eqref{Mlimite}, we then write
 \begin{equation}
 \label{F1NLdef}
 \mathcal{F}_1(t)
 =\sum_{\ell\neq 0, \, k}  \mathcal{J}_{k,\ell} (t)
 \end{equation}
 where
 $$  \mathcal{J}_{k,\ell} (t)=   -\eps \int_0^t \zeta(s) 
  \int_{a} g_{k+\ell}(s,a)  \mathcal{A}_{k,\ell} e^{i( \ell t + ks) \omega_{M_\eps} }\, \dd a  \dd s
 = -\eps \int_0^t  \zeta(s)  I[ g_{k+\ell} \mathcal{A}_{k,\ell}, \omega_{M_\eps}](\ell t+ks)\, \dd s
 $$ 
and 
 \begin{multline*}
  \mathcal{A}_{k,\ell}(t,s,a)
 =   \left( ik C_{k,M(s)} C_{\ell,M(t)}' - i\ell C_{k,M(s)}' C_{\ell,M(t)} - k\ell C_{k,M(s)} C_{\ell,M(t)} \int_s^t \omega_{M(\tau)}'\right)\\ \times 
   e^{ i\ell\int_0^t( \omega_M - \omega_{M_\eps})\,\dd \tau  + ik \int_0^s( \omega_M - \omega_{M_\eps})\,\dd \tau }.
 \end{multline*}
 We shall thus estimate $\mathcal{M}_{T, \sigma -1} (\mathcal{J}_{k,\ell})$ for $\ell\neq 0$.
 
By using again Lemma \ref{time-decay-ipp} and  again thanks to the fast decay of the Fourier coefficients 
 $C_{n, M}$, we obtain for every $p \geq 0$ and $B\geq 0$ that   
\begin{equation}
\label{estNLecho} | I[ g_{k+\ell} \mathcal{A}_{k,\ell}, \omega_{M_\eps}](\ell t+ks)| \lesssim_{p,B} {\Norm{ g_{k+\ell}(s)}{H^p}  \over \langle \ell t + ks \rangle^p}  { \langle t-s\rangle \over \langle \ell\rangle^B \langle k\rangle^{B}}\lesssim_{p,B} {\Norm{ g(s)}{H^p}  \over \langle \ell t + ks \rangle^p}  { \langle t-s\rangle \over \langle \ell\rangle^B \langle k\rangle^{B}} .
\end{equation}
 Note that this estimate is uniform in $t$, $s$ and $k$, $\ell$.
 
  To estimate $\mathcal{M}_{T, \sigma -1} (\mathcal{J}_{k,\ell})$ for $\ell\neq 0$, we discuss according to the relative signs of $\ell$ and $k$.
  \begin{itemize}
  \item If $\ell k\geq0$ (which is equivalent to  $\ell k>0$  or $k=0$).
  
   We first  use \eqref{estNLecho} with $p=\sigma$ so that for $t <T^\eps$, we have
   $$   | I[ g_{k+\ell} \mathcal{A}_{k,\ell}, \omega_{M_\eps}](\ell t+ks)| \lesssim_{B}
     {\Norm{ g(s)}{H^\sigma}  \over \langle \ell t + ks \rangle^\sigma}  { \langle t-s\rangle \over \langle \ell\rangle^B \langle k\rangle^{B}}
     \lesssim_B \mathcal{R} { \langle s\rangle^3 \over\langle t\rangle ^{\sigma-1}} { 1  \over \langle \ell\rangle^B \langle k\rangle^{B}}$$
     where for the last estimate, we have used that from our bootstrap assumption, we have for $t<T^\eps$
     $$ \Norm{ g(s)}{H^\sigma}\leq \mathcal{R} \langle s \rangle^3.$$ 
     Note that it is crucial for the argument that $\ell\neq 0$.
     We thus obtain that
     \begin{equation}
     \label{estJklpasdecho} 
     \mathcal{M}_{T, \sigma -1} (\mathcal{J}_{k,\ell}) \lesssim_B  \eps \mathcal{R}  \mathcal{M}_{T, \sigma -1} (\zeta)  { 1  \over \langle \ell\rangle^B \langle k\rangle^{B}}
       \int_0^t { 1 \over \langle s\rangle^{\sigma-4}} \, \dd s\lesssim_B  \eps \mathcal{R}  \mathcal{M}_{T, \sigma -1} (\zeta)  { 1  \over \langle \ell\rangle^B \langle k\rangle^{B}}
       \end{equation}
       assuming that $\sigma \geq 6$.
       
       \item If $\ell k<0$.
       
       We are in the regime where echoes are possible. Indeed  if $|k| > | \ell |$,  for $s=-\ell t/k$, even for large $p$, the right hand side of  \eqref{estNLecho} 
       behaves like $  (1+ {\ell\over k})t$ which is very large for large times except if $\ell=-k$.
        We  thus define
        $$ t_{k,\ell} = { t \over 2} \min (- { \ell \over k}, 1).$$
        For $s \leq t_{k,\ell}, $ we use again  \eqref{estNLecho} with $p=\sigma$, we obtain
        $$ | I[ g_{k+\ell} \mathcal{A}_{k,\ell}, \omega_{M_\eps}](\ell t+ks)| \lesssim_{B}  \mathcal{R} \langle s \rangle^3  
        {  1 \over \langle \ell\rangle^B \langle k\rangle^{B}} { \langle t-s \rangle \over  \langle \ell t + ks \rangle^\sigma}$$
        and we observe that for $s \leq t_{k,\ell}, $ we have
        $$ { 1  \over  \langle \ell t + ks \rangle^\sigma} \lesssim  { 1 \over  \langle t \rangle^{\sigma}}.$$
        This yields as previously, 
      \begin{multline}
      \label{debutecho}
      \eps \Big| \int_0^{t_{k,\ell}}  \zeta(s)  I[ g_{k+\ell} \mathcal{A}_{k,\ell}, \omega_{M_\eps}](\ell t+ks)\, \dd s   \Big|
       \lesssim_B \eps \mathcal{R}    \mathcal{M}_{T, \sigma -1} (\zeta)  { 1  \over \langle \ell\rangle^B \langle k\rangle^{B}}
        \int_0^{t_{k,\ell}}  { 1 \over \langle s\rangle^{\sigma -4}}\, \dd s  \\
         \lesssim_B \eps \mathcal{R}    \mathcal{M}_{T, \sigma -1} (\zeta)  { 1  \over \langle \ell\rangle^B \langle k\rangle^{B}}.
         \end{multline}
         It remains $s\geq t_{k,\ell}$
   For $s\geq t_{k,\ell}$,   we use \eqref{estNLecho} for $p=\sigma- 4$, this yields
   $$     | I[ g_{k+\ell} \mathcal{A}_{k,\ell}, \omega_{M_\eps}](\ell t+ks)| \lesssim_{B}
     {\Norm{ g(s)}{H^{\sigma- 4}}  \over \langle \ell t + ks \rangle^{\sigma- 4 } }  { \langle t-s\rangle \over \langle \ell\rangle^B \langle k\rangle^{B}}
     \lesssim_B \mathcal{R} {   \langle t-s \rangle  \over  \langle \ell t + ks \rangle^{\sigma- 4 } } { 1  \over \langle \ell\rangle^B \langle k\rangle^{B}}$$
     where we have now used that for $t <T^\eps$, we have $ \Norm{ g(t)}{H^{\sigma- 4}} \leq \mathcal{R}.$
      Since $k\neq 0$,  we can use that 
      $$ {   1   \over  \langle \ell t + ks \rangle^{\sigma- 4 } }  \lesssim  { 1 \over \langle s - \left| {k\over \ell }  \right| t  \rangle^{\sigma -4}}$$
      and that $t-s \leq t$ to obtain
      $$  | I[ g_{k+\ell} \mathcal{A}_{k,\ell}, \omega_{M_\eps}](\ell t+ks)| \lesssim_{B} { 1  \over \langle \ell\rangle^B \langle k\rangle^{B}}
       { t \over \langle s - \left| {k\over \ell}\right| t  \rangle^{\sigma -4}}.$$
       Since we have 
       $| \zeta(s)| \leq { \mathcal{R} \over \langle s\rangle^{\sigma-1}}, $
       we get for $s\geq t_{k,\ell} $ that
       \begin{equation}
       \label{lapasbonne}
       | \zeta(s)| \lesssim  { \mathcal{R} \over \langle t\rangle^{\sigma-1} } \max \left( 1 ,  \left| {k\over \ell }\right|^{ \sigma -1} \right).
        \end{equation}
        Gathering the previous estimates, we then obtain that
        \begin{align}
        \label{finecho}
      \eps &\Big| \int_{t_{k,\ell}}^{t}  \zeta(s)  I[ g_{k+\ell} \mathcal{A}_{k,\ell}, \omega_{M_\eps}](\ell t+ks)\, \dd s   \Big| \\&\nonumber
       \lesssim_B \eps \mathcal{R} \langle t \rangle { 1 \over \langle t \rangle^{\sigma -1}}     \mathcal{M}_{T, \sigma -1} (\zeta)  { 1  \over \langle \ell\rangle^B \langle k\rangle^{B- \sigma + 1}}
        \int_\mathbb{R}  { 1 \over \langle s - \left| {k\over \ell }  \right| t  \rangle^{\sigma -4}} \, \dd s   \\&
         \lesssim_B   \langle t \rangle { 1 \over \langle t \rangle^{\sigma -1}}     \mathcal{M}_{T, \sigma -1} (\zeta)  { 1  \over \langle \ell\rangle^B \langle k\rangle^{B- \sigma + 1}}\nonumber
         \end{align}
         since $\sigma >5$.
         
  By combining \eqref{debutecho} and \eqref{finecho}, we thus get
  \begin{equation}
  \label{estJklecho}
   \mathcal{M}_{T, \sigma -1} (\mathcal{J}_{k,\ell}) \lesssim_B  \eps \mathcal{R}  \langle T\rangle \mathcal{M}_{T, \sigma -1} (\zeta)  { 1  \over \langle \ell\rangle^B \langle k\rangle^{B-\sigma+ 1}}.
  \end{equation}
  \end{itemize}
 
 Consequently, from \eqref{estJklecho}, \eqref{estJklpasdecho}, we obtain by summing up on $k$, $\ell$  for $B$ sufficiently large 
 \begin{equation}
 \label{estF1NL}
 \mathcal{M}_{T, \sigma -1} (\mathcal{F}_1) \lesssim   \eps \mathcal{R}  \langle T\rangle \mathcal{M}_{T, \sigma -1} (\zeta).
 \end{equation}
 
 To end the proof of Proposition \ref{propzetaNL} it remains to estimate $\mathcal{F}_2$  defined in \eqref{FNL2}.
  The estimate follows the lines  of  the one of $\mathcal{F}_1$ since the structure of $\mathcal{F}_2$ is very similar, 
  we just replace $\zeta$ by $\dot m$ which has a faster decay  and the Fourier coefficients of $\cos x_{M(s)}$
  by the Fourier coefficients of $K_{M(s)}$ that we denote by $K_{\ell,M} $. Since $K_{M}$  is still a given function smooth on the support of $g$ its Fourier coefficients  still enjoy fast decay.
 We can expand 
 \begin{equation}
 \label{F2NLdef}
 \mathcal{F}_2(t)
 =\sum_{\ell\neq 0, \, k} \widetilde{\mathcal{J}}_{k,\ell} (t)
 \end{equation}
 where
 $$  \widetilde{\mathcal{J}}_{k,\ell} (t)=   -\eps \int_0^t \dot m(s) 
  \int_{a} g_{k+\ell}(s,a)  \tilde{\mathcal{A}}_{k,\ell} e^{i( \ell t + ks) \omega_{M_\eps} }\, \dd a  \dd s
 = -\eps \int_0^t  \dot m (s)  I[ g_{k+\ell} \widetilde{\mathcal{A}}_{k,\ell}, \omega_{M_\eps}](\ell t+ks)\, \dd s
 $$ 
and 
 \begin{multline*}
 \widetilde{\mathcal{A}}_{k,\ell}(t,s,a)
 =  \left( ik K_{k,M(s)} C_{\ell,M(t)}' - i\ell K_{k,M(s)}' C_{\ell,M(t)} - k\ell K_{k,M(s)} C_{\ell,M(t)} \int_s^t \omega_{M(\tau)}'\right)\\ \times
   e^{ i\ell \int_0^t( \omega_M - \omega_{M_\eps})\,\dd \tau  + ik \int_0^s( \omega_M - \omega_{M_\eps})\,\dd \tau }.
 \end{multline*}
 In order to estimate $\dot m$, we use Proposition \ref{propdotm}. 
 We can thus proceed exactly as in the estimate of $\mathcal{F}_1.$
 The only slight difference is in the regime $\ell k<0$ for $s \geq t_{k,\ell}$. Indeed, by using Proposition \ref{propdotm}, 
 assuming that $2\sigma - 4\geq \sigma$
 we get instead of \eqref{lapasbonne}
  that
$$
       | \dot m (s)| \lesssim  {\eps  \mathcal{R} \over \langle t\rangle^{\sigma} } \max ( 1 ,  \left| {k\over \ell }\right|^{ \sigma} ) \mathcal{M}_{T, \sigma-1} (\zeta).
$$
 This allows to avoid loosing the factor  $T$ in the final estimate.
 We finally obtain that
 \begin{equation}
 \label{estF2NL}
  \mathcal{M}_{T, \sigma -1} (\mathcal{F}_2) \lesssim   \eps^2 \mathcal{R}^2  \mathcal{M}_{T, \sigma -1} (\zeta).
 \end{equation}

We end the proof of Proposition \ref{propzetaNL} by gathering \eqref{proofzetaNL1},  \eqref{FNLsplit} and \eqref{estFI}, \eqref{estFm}, \eqref{estF1NL}, 
\eqref{estF2NL}. This yields
$$ \mathcal{M}_{T, \sigma -1} (\zeta)   \leq  C\left( \mathcal{R}_0 + \mathcal{R}^2 \eps)  +  \mathcal{R}^2 \eps \langle T\rangle  \mathcal{M}_{T, \sigma -1} (\zeta) \right)
  .$$
We conclude by taking $ C \mathcal{R}^2 \eps \langle T\rangle \leq 1/2.$

\end{proof}

\subsection{Energy estimates}
We shall finally estimate the Sobolev norms of $g$ namely $\Norm{ g }{H^\sigma} $ and  $\Norm{ g }{H^{\sigma- 4}} $. 

\begin{Proposition}
\label{propenergie}
There exists $C>0$  such that for every  $\eps \in (0, \eps_0]),$ every  $ T \in [0, T^\eps)$ satisfying \eqref{Tcontrainte}, we have the estimate
$$ \Norm{ g(t) }{H^{\sigma -4}} + { \Norm{g(t) }{H^\sigma} \over \langle t \rangle^3} \lesssim 
(\mathcal{R}_0 + \mathcal{R}^2 \eps) e^{C\eps  ( \mathcal{R}_0 + \eps \mathcal{R}^2 )}\Big(  1 + \eps e^{C \eps ( \mathcal{R}_0 + \eps \mathcal{R}^2 )}( \mathcal{R}_0 + \eps \mathcal{R}^2)\Big) .$$

\end{Proposition}

  \begin{proof}
  We shall perform energy estimates on \eqref{eqg} that we rewrite under the form 
  $$  \partial_{t} g  - \eps \zeta (t)   \{ g, \cos x_{M(t)} \circ P_{t}^+ \}_{\theta, a} 
-  \eps \dot m \{ g, K_{M(t)} \circ P_{t}^+ \}_{\theta, a}
  =\mathcal{F}_g.$$
  where 
  $$  \mathcal{F}_g =  \zeta (t)  \{ \Gamma_{M(t)}, \cos x_{M(t)} \circ P_{t}^+ \}_{\theta, a}  -  \dot m  \Gamma_{M(t)}^1 \circ P_t^+.$$
  
  For every $p\geq 0$ and $|\alpha|\leq p$, we thus obtain that
  $$ \partial_{t} \partial^\alpha g  - \eps \zeta (t)   \{ \partial^\alpha g, \cos x_{M(t)} \circ P_{t}^+ \}_{\theta, a} 
-  \eps \dot m \{ \partial^\alpha g, K_{M(t)} \circ P_{t}^+ \}_{\theta, a}
  =\mathcal{F}_g + \mathcal{C}_{p, g}$$
  where for $p\geq 1$, the commutator $\mathcal{C}_{p,g}$ is given by
  $$  \mathcal{C}_{p,g} = \sum_{|\beta| <|\alpha| } C_{\alpha, \beta} \left( \eps \zeta (t)   \{ \partial^\beta g,\partial^{\alpha - \beta} ( \cos x_{M(t)} \circ P_{t}^+) \}_{\theta, a}
  +   \eps \dot m \{ \partial^\beta g, \partial^{\alpha - \beta } (K_{M(t)} \circ P_{t}^+) \}_{\theta, a}\right).$$ 
 Since, $g$ is compactly supported in $\mathbb{T} \times [a_-, a_+] $, we have by integration by parts that
 $$ \int_{\theta, a}    \{ \partial^\alpha g, \cos x_{M(t)} \circ P_{t}^+ \}_{\theta, a}  \partial^\alpha g \, \dd \theta \dd a= 
 \int_{\theta, a}    \{ \partial^\alpha g, K_{M(t)}\circ P_{t}^+ \}_{\theta, a}  \partial^\alpha g \, \dd \theta \dd a = 0$$
 and hence 
 $$ {d \over dt } \Norm{ \partial^\alpha g }{L^2}^2 \lesssim  \Norm{g }{H^p} \Norm{\mathcal{C}_{p,g} }{L^2}  + \Norm{g }{H^p} \Norm{\mathcal{F}_g}{W^{p, \infty}(\mathbb{T} \times [a_-, a_+])}.$$
 Next, we observe that for every nonnegative  $k$ and $\ell$, we have
\begin{multline}
\label{estcircP}
\Norm{  \partial_\theta^\ell( \cos x_{M(t)} \circ P_{t}^+)}{W^{k, \infty}(\mathbb{T} \times [a_-, a_+])}
 + \Norm{ K_{M(t)}\circ P_{t}^+ }{W^{k, \infty}(\mathbb{T} \times [a_-, a_+])}  \\ + \Norm{ \Gamma_{M(t)}^1\circ P_{t}^+ }{W^{k, \infty}(\mathbb{T} \times [a_-, a_+])} \lesssim \langle t \rangle^k .
 \end{multline}
 Since
 $$  \mathcal{F}_g =  -  \zeta (t)  \Gamma_{M(t)}' \partial_\theta  [\cos x_{M(t)} \circ P_{t}^+ ] -  \dot m  \Gamma_{M(t)}^1 \circ P_t^+, $$
 this yields by using  \eqref{lestimdezeta} and Proposition \ref{propdotm}
 \begin{equation}
  \label{energiep}
  {d \over dt } \Norm{ \partial^\alpha g }{L^2}^2 \lesssim  \Norm{g }{H^p} \Norm{\mathcal{C}_{p,g} }{L^2}  + \Norm{g }{H^p}(\mathcal{R}_0 + \mathcal{R}^2 \eps) { \langle t \rangle^{p} \over \langle t \rangle^{\sigma -1}}.
  \end{equation}
  It remains to estimate $\mathcal{C}_{p,g}$.
  
  For $p= \sigma- 4$, we directly get that
  \begin{multline*}
   \Norm{\mathcal{C}_{\sigma -4,g} }{L^2} \lesssim \eps \Norm{ g }{H^{\sigma -4}} \Big[| \zeta(t)|  \Norm{ \cos x_M \circ P_t^+}{W^{\sigma -3, \infty}(\mathbb{T} \times [a_-, a_+])} 
   + |\dot m (t) |  \Norm{ K_{M(t)}\circ P_{t}^+ }{W^{\sigma-3, \infty}(\mathbb{T} \times [a_-, a_+])} \Big] \\
   \lesssim \eps \Norm{ g }{H^{\sigma -4}} ( \mathcal{R}_0 + \eps \mathcal{R}^2 ) { 1 \over \langle t \rangle^2}.
   \end{multline*}
   Therefore, we get from \eqref{energiep} that
   $$ { d \over dt } \Norm{ g(t)}{H^{\sigma-4}} \lesssim  \eps ( \mathcal{R}_0 + \eps \mathcal{R}^2 )  { 1 \over \langle t \rangle^2}  \Norm{g(t)}{H^{\sigma-4}} 
    + ( \mathcal{R}_0 + \eps \mathcal{R}^2 )  {1\over \langle t \rangle^3}.$$
This yields after integration 
\begin{equation}
\label{preuvesigma-41}
  \Norm{ g(t)}{H^{\sigma-4}} \lesssim e^{C  \eps ( \mathcal{R}_0 + \eps \mathcal{R}^2 )}( \mathcal{R}_0 + \eps \mathcal{R}^2).
  \end{equation}

To estimate $\Norm{g}{H^\sigma}$, we use \eqref{energiep} for $p=\sigma$. To estimate the commutator, we use again \eqref{estcircP} 
and we single out in the sum the terms with $|\alpha|-1 \geq |\beta|\geq |\alpha|- 5$, this yields
$$  \Norm{\mathcal{C}_{\sigma, g}}{L^2} \lesssim \eps \big(|\zeta(t)| + |\dot m(t) |\big) \Big[ 
 \Norm{g(t)}{H^{\sigma-4}} \langle t \rangle^{ \sigma +1}
  +\sum_{ k=0}^3 \Norm{g(t)}{H^{\sigma-k}} \langle t \rangle^{  k+2}\Big]. 
  $$
  From the interpolation inequality 
  $$ \Norm{ g}{H^{\sigma -k}} \lesssim   \Norm{g} {H^{\sigma -4}}^{ k\over 4} \Norm{g}{H^{\sigma}}^{ 1 - { k\over 4}}, \quad 1 \leq k \leq 3,  $$
  we obtain 
  $$ \langle t \rangle^{k+2} \Norm{g}{H^{\sigma -k}} \lesssim  \big[\langle t \rangle^6 \Norm{ g}{H^{\sigma -4}} \big]^{ k\over 4} \big[ \langle t \rangle^2 \Norm{ g}{H^{\sigma}}^{ 1 - { k\over 4}}\big], \quad 1 \leq k \leq 3,  $$
  and hence from the Young inequality (and since  $\sigma \geq 5$),  we obtain
  $$  \Norm{\mathcal{C}_{\sigma, g} }{L^2} \lesssim \eps \big(|\zeta(t)| + |\dot m(t) |\big) \Big( 
 \Norm{g(t)}{H^{\sigma-4}} \langle t \rangle^{ \sigma +1}
  + \Norm{g(t)}{H^{\sigma}} \langle t \rangle^{ +2}\Big). 
  $$
  We thus get from \eqref{energiep}  and  by using again   \eqref{lestimdezeta} and Proposition \ref{propdotm}  that
  $$ { d \over dt } \Norm{ g(t)}{H^{\sigma}} \lesssim   (\mathcal{R}_0 + \mathcal{R}^2 \eps)  \langle t \rangle^{2}( 1 +  \eps \Norm{g(t)}{H^{\sigma-4}})
   +  \eps (\mathcal{R}_0 + \mathcal{R}^2 \eps)  { 1\over \langle t \rangle^{\sigma - 3} } \Norm{ g(t)}{H^{\sigma}}.$$
   By finally using \eqref{preuvesigma-41}, we get from the Gronwall inequality that
   $$  \Norm{ g(t)}{H^{\sigma}} \lesssim   (\mathcal{R}_0 + \mathcal{R}^2 \eps) e^{C\eps  ( \mathcal{R}_0 + \eps \mathcal{R}^2 )}\Big(  1 + \eps e^{C \eps ( \mathcal{R}_0 + \eps \mathcal{R}^2 )}( \mathcal{R}_0 + \eps \mathcal{R}^2)\Big)  \langle t \rangle^{3}.$$
   This ends the proof.
   \end{proof}
   
   \subsection{End of the proof of Theorem \ref{thnonlin} }
   \label{cestfini}
   
   By gathering the results of  Proposition \ref{propenergie} and Proposition \ref{propzetaNL}, we obtain that
   $$  Q_\sigma (T) \leq C
(\mathcal{R}_0 + \mathcal{R}^2 \eps) e^{C\eps  ( \mathcal{R}_0 + \eps \mathcal{R}^2 )}\Big[  1 + \eps e^{C \eps ( \mathcal{R}_0 + \eps \mathcal{R}^2 )}( \mathcal{R}_0 + \eps \mathcal{R}^2)\Big] $$
   for $T\leq T^\eps$ as long as \eqref{Tcontrainte} is satisfied and some harmless number $C>0$.
   We thus fix $\mathcal{R}  = 4 \mathcal{R}_0 C$ and take $\eps$ sufficiently small to obtain that 
   $$  Q_\sigma (T) \leq { \mathcal{R} \over 2}.$$
   By setting $T_0= { 1 \over  2 C_0 \eps \mathcal{R}^2}.$ We then deduce by a standard continuity argument that
   $ T^\eps  =\min( T^*, T_0)$  for $T^\eps$ defined in \eqref{Tepsdef}  and $T^*$ in \eqref{T*def}.
   We shall finally prove that $T^*\geq T_0$ to finish the proof.
   
   From Proposition \ref{propdotm} and \eqref{m0petit}  we obtain for $T \in (0,  \min( T^*, T_0))$
      \begin{equation}
      \label{Mconclusion}
       |M(t) -  M_0| \leq |M(t) - M(0)|+ |M(0)- M_0| \leq C_0 \eps  +  C \mathcal{R}^2 \eps^2 \leq { 3 C_0\over 2} 
       \end{equation}
      by taking $\eps_0$ smaller.
      Moreover, thanks to \eqref{decdyn}, \eqref{defpetit},  \eqref{defg}, \eqref{deftildeg}, we have
      $$ (f(t) - \eta_{M(t)} )\circ \Phi_{M(t)}  = \eps g(t) \circ  P_t^-$$
      and thus 
      $$ \Norm{f(t) - \eta_{M(t)}}{L^\infty} \lesssim \eps \Norm{ g }{L^\infty} \lesssim \eps \mathcal{R}.$$
      By using again \eqref{Mconclusion}, this yields
      \begin{equation}
      \label{fconclusion}
       \Norm{ f(t) - \eta_{M_0}}{L^\infty} \lesssim \eps  +  \eps \mathcal{R} + \eps^2 \mathcal{R}^2, \quad T \in \big(0,  \min( T^*, T_0)\big).
      \end{equation}
      Finally we can check our assumption on the support of $f -\eta_{M_{0}}$ in the definition of $T^*$.
      Since $f$ solves \eqref{eq:VlasovHMFmodel},  we have by the method of characteristics
      \begin{equation}
      \label{charsolve}
       f(t,x,v)= f^0( \Theta_{0, t} (x,v))
      \end{equation}
      where $\Theta_{s,t}= (X_{s,t}, V_{s,t})$ are the backward characteristics  solving
      $$
      \left|
      \begin{array}{l}  
      \partial_s X_{s,t} = V_{s,t},\\
      \partial_s V_{s,t} =  - \mathcal{C} [f](s) \sin X_{s,t}= - (M(s) + \eps \zeta(s) ) \sin X_{s,t}
      \end{array}
      \right.
      $$
      with $  \Theta_{t,t} (x,v) = (x,v)$.
      Let us set $h_{s,t}(x,v) = h_{M(s)} [\Theta_{s,t} (x,v)]$. We get that
      $$ \partial_s h_{s,t} =  - \dot M (s) \cos (X_{s,t}) - \eps \zeta (s) V_{s,t}  \sin X_{s,t}$$
      and thus we obtain
      $$| \partial_s h_{s,t} | \leq  \Big[|\dot M(s)|  + \eps| \zeta(s)| (1 + |M(s)|) \Big] +   \eps |\zeta (s)|  | h_{s,t}|.$$
      By using again \eqref{lestimdezeta}, and  Proposition \ref{propdotm}, we then find after integration in time that $y_{s,t}(x,v)  = |h_{s,t}(x,v) - h_{M(t)}(x,v)|$ satisfies 
      $$ y_{s,t} \leq  \eps \Lambda_{\mathcal{R}} + \int_s^t  \eps |\zeta (\tau)| y(\tau)\, \dd \tau
       \lesssim \eps \Lambda_{\mathcal{R}}  + \eps \mathcal{R} \sup_{0 \leq s \leq t} y_{s,t}.$$
       for some harmless number $\Lambda_\mathcal{R}$ depending only on $\mathcal{R}$.
     This yields for $\eps \mathcal{R}$ sufficiently small
     $$ y_{s,t}(x,v) \lesssim  \eps \Lambda_{\mathcal{R}}.$$
     Note that this estimate is uniform in $x$, $v$ and for $0 \leq s\leq t \leq  \min( T^*, T_0).$
     We have thus obtained in particular that
     $$| h_{0, t} (x,v) -  h_{M(t)}(x,v)| \lesssim  \eps \Lambda_{\mathcal{R}}.$$
     By  using again \eqref{Mconclusion}, we further get 
     $$ | h_{0, t} (x,v) -  h_{M_0}(x,v)| \lesssim  \eps \Lambda_{\mathcal{R}}.$$
     In particular, for $\eps$ sufficiently small,  we can ensure that
     $$ | h_{0, t} (x,v) -  h_{M_0}(x,v)| \lesssim  \kappa_0$$
     and thus,  if $h_{M_0}(x,v) \geq  M_0- 2 \kappa_0$,  $h_{0, t} (x,v)  \geq  M_0 - 3\kappa_0\geq M_0- \overline{\kappa_0}$.
      This yields  using \eqref{charsolve} and the support assumption for  the initial data and the stationary state that
      $ f(t,x,v)= 0.$
      In a similar way, if  $h_{M_0}(x,v) \leq  -M_0+  2\kappa_0$, we have $ h_{0, t} (x,v) \leq  - M_0 + \overline{\kappa_0}$. Moreover,  since
      $$ f(t,x,v) - \eta_{M_0} (x,v) = \eta_{M_0} (\Theta_{0, t}(x,v)) - \eta_{M_0} (x,v) + \eps r^0 ( \Theta_{0, t}(x,v)), $$
      we deduce by the support assumption for $r^0$ that $ r^0 ( \Theta_{0, t}(x,v))=0$. Moreover, thanks to (H1), $\eta_{M_0}$ is constant in 
       $h_{M_0} \leq  -M_0 + \overline{\kappa}_0$ and thus  $\eta_{M_0} (\Theta_{0, t}(x,v)) - \eta_{M_0} (x,v) =0$.
       This proves that $ f(t,x,v) - \eta_{M_0} (x,v) $ vanishes if  $h_{M_0}(x,v) \leq  -M_0+ 2\kappa_0$.
       
       Since we have now established that for $t \in   (0,  \min( T^*, T_0))$, we have the estimates  \eqref{fconclusion}, \eqref{Mconclusion} 
        and that $f(t) - \eta_{M_0}$ is compactly supported in $\{ - M_0 + 2 \kappa_0 \leq h_{M_0} \leq M_0 - 2 \kappa_0\}$ which
        is strictly included in $\mathcal{K}_{M_0}$, we obtain by a standard continuity argument that for $ \eps$ sufficiently small we must
         have $T^* \geq T_0$. This ends the proof.

\thanks{{\bf Statement}: No conflict of interest. Data availability: not relevant.}


\end{document}